\let\savedVert\|

\makeatletter
\def\cup@reference@code{%
  \RequirePackage[style=numeric,sorting=none,backend=biber]{biblatex}%
  \renewcommand*{\bibfont}{\footnotesize}%
}
\makeatother

\PassOptionsToPackage{table}{xcolor}

\documentclass[
  journal=largetwo,
  manuscript=article,
  manuscriptlabel={Research Preprint},
  logo=false,
  simfonts=false,
  year=2026,
  volume=1,
]{cup-journal}

\let\|\savedVert

\usepackage[T1]{fontenc}
\usepackage{amsmath,amssymb,amsthm}
\usepackage{mathtools}
  
\usepackage[varg]{newtxmath}
\DeclareMathSizes{10.95}{9.95}{7}{5}
\usepackage{graphicx}
\usepackage{booktabs}
\usepackage{multirow}
\usepackage{makecell}
\usepackage{enumitem}
\usepackage{needspace}
\setlist{nosep,leftmargin=*}
\usepackage{xcolor}
\usepackage{microtype}
\usepackage[colorlinks=true,linkcolor=blue,citecolor=blue]{hyperref}
\hypersetup{
  pdftitle={A Second-Moment Theory for Floating-Point Reduction Trees},
  pdfauthor={Piyush Sao; Narasinga Miniskar; Pedro Valero-Lara; Keita Teranishi; Sudip Seal},
  pdfsubject={Second-moment analysis of floating-point reduction trees},
  pdfkeywords={floating-point arithmetic, reduction tree, rounding error, variance analysis, numerical stability}
}
\usepackage{cleveref}
\usepackage{orcidlink}

\usepackage{etoolbox}
\AtBeginEnvironment{table}{\footnotesize\setlength{\tabcolsep}{3pt}}
\newcommand{\compactdisplayspacing}{%
  \setlength{\abovedisplayskip}{8pt plus 2pt minus 3pt}%
  \setlength{\belowdisplayskip}{8pt plus 2pt minus 3pt}%
  \setlength{\abovedisplayshortskip}{0pt plus 2pt}%
  \setlength{\belowdisplayshortskip}{5pt plus 2pt minus 2pt}%
}
\AtBeginDocument{\compactdisplayspacing}
\AtBeginEnvironment{align*}{\small\compactdisplayspacing}

\definecolor{primary}{HTML}{006BA2}
\definecolor{accent}{HTML}{E3120B}
\definecolor{neutral}{HTML}{4A4A4A}

\definecolor{plotblue}{HTML}{006BA2}
\definecolor{plotred}{HTML}{E3120B}
\definecolor{plotgreen}{HTML}{2E7D32}
\definecolor{plotorange}{HTML}{FF6F00}

\definecolor{boxfill}{HTML}{E8F4FD}
\definecolor{boxborder}{HTML}{006BA2}

\definecolor{hlblue}{HTML}{E3F2FD}
\definecolor{hlred}{HTML}{FFEBEE}

\newtheorem{theorem}{Theorem}[section]

\newtheorem{proposition}[theorem]{Proposition}
\newtheorem{corollary}[theorem]{Corollary}

\theoremstyle{definition}
\newtheorem{definition}[theorem]{Definition}

\theoremstyle{remark}

\newtheorem{assumption}{Assumption}

\usepackage[most]{tcolorbox}

\definecolor{theoremgreen}{HTML}{2E7D32}
\definecolor{mathboxgray}{HTML}{6B7280}

\tcbset{mathboxbase/.style={%
  enhanced, breakable, sharp corners, frame hidden, boxrule=0pt,
  left=10pt, right=6pt, top=5pt, bottom=5pt,
  before skip=8pt, after skip=8pt},
  theorembox/.style={%
    mathboxbase, colback=theoremgreen!6,
    borderline west={3pt}{0pt}{theoremgreen}},
  neutralmathbox/.style={%
    mathboxbase, colback=white, opacityback=0,
    borderline west={2pt}{0pt}{mathboxgray}}}

\tcolorboxenvironment{theorem}{theorembox}
\tcolorboxenvironment{lemma}{neutralmathbox}
\tcolorboxenvironment{proposition}{neutralmathbox}
\tcolorboxenvironment{corollary}{neutralmathbox}
\tcolorboxenvironment{definition}{neutralmathbox}
\tcolorboxenvironment{example}{neutralmathbox}
\tcolorboxenvironment{problem}{neutralmathbox}
\tcolorboxenvironment{assumption}{neutralmathbox}
\tcolorboxenvironment{remark}{neutralmathbox}
\tcolorboxenvironment{note}{neutralmathbox}

\AtBeginBibliography{\sloppy}

\newcommand{\E}{\mathbb{E}}

\newcommand{\ra}[1]{\renewcommand{\arraystretch}{#1}}

\makeatletter
\def\cup@journal@name{Floating-Point Reduction Trees}
\def\cup@manuscript{preprint}
\def\cup@year{July 2026}
\makeatother

\title{A Second-Moment Theory for Floating-Point\\
  Reduction Trees}
\author{Piyush Sao\,\orcidlink{0000-0002-9432-5855}}
\affiliation{Oak Ridge National Laboratory, Oak Ridge, TN, USA}
\email[Piyush Sao]{saopk@ornl.gov}
\author{Narasinga Miniskar\,\orcidlink{0000-0001-8259-8891}}
\affiliation{Oak Ridge National Laboratory, Oak Ridge, TN, USA}
\email[Narasinga Miniskar]{miniskarnr@ornl.gov}
\author{Pedro Valero-Lara\,\orcidlink{0000-0002-1479-4310}}
\affiliation{Oak Ridge National Laboratory, Oak Ridge, TN, USA}
\email[Pedro Valero-Lara]{valerolarap@ornl.gov}
\author{Keita Teranishi\,\orcidlink{0000-0001-6647-2690}}
\affiliation{Oak Ridge National Laboratory, Oak Ridge, TN, USA}
\email[Keita Teranishi]{teranishik@ornl.gov}
\author{Sudip Seal\,\orcidlink{0000-0003-3233-0656}}
\affiliation{Oak Ridge National Laboratory, Oak Ridge, TN, USA}
\email[Sudip Seal]{sealsk@ornl.gov}
\makeatletter
\def\cup@contact@details{%
  {\emailfont\textbf{Emails (author order): }
  \href{mailto:saopk@ornl.gov}{saopk@ornl.gov};
  \href{mailto:miniskarnr@ornl.gov}{miniskarnr@ornl.gov};
  \href{mailto:valerolarap@ornl.gov}{valerolarap@ornl.gov};
  \href{mailto:teranishik@ornl.gov}{teranishik@ornl.gov};
  \href{mailto:sealsk@ornl.gov}{sealsk@ornl.gov}.}%
}
\makeatother

\begin{document}

\begin{abstract}
Summation error depends on partial-sum order, which standard worst-case bounds omit. To capture this dependence, we derive an exact mean-square error (MSE) recurrence for a binary tree $T$ under conditionally unbiased rounding. With unit roundoff $u$, the constant-$\nu$ model (Section~\ref{sec:constant-nu}) sets the local variance at the pre-rounding value $x$ to $\nu u^2x^2$. Its leading tree-dependent cost for the input vector $p$ is $p^\top K_Tp$, where the common-ancestor kernel $K_T$ (Section~\ref{sec:ancestor-kernel}) counts internal ancestors shared by leaves $i$ and $j$. For i.i.d. inputs of mean $\mu$ and variance $\tau^2$, this expected cost is $\tau^2\Lambda_1(T)+\mu^2\Lambda_2(T)$; the scalar tree statistics $\Lambda_1$ and $\Lambda_2$ are defined in Section~\ref{sec:iid-compression}. Here $\Lambda_1$ is total leaf depth, while $\Lambda_2$ sums squared internal-subtree sizes. Thus $\Lambda_1$ governs centered inputs, while $\Lambda_2$ captures nonzero means.

We use these statistics to characterize optimal tree topologies and schedules. Balanced and sequential trees attain the centered extrema. For $k$ inputs, optimal two-stage sequential blocking yields root-mean-square (RMS) error that scales as $k^{3/4}$. For fixed-stage hierarchies, geometric schedules are optimal for centered inputs, whereas the optimal noncentered stage exponents halve successively. For independent centered inputs with unequal variances, Huffman coding minimizes variance-weighted depth over free leaf assignments and unconstrained trees. We extend the kernel to matrix multiplication through operand Gram matrices.

To evaluate these predictions, we test the approximation under round-to-nearest using exact residuals. Across binary64, binary32, and software-emulated binary16 and bfloat16, the model recovers the ordering among tree topologies; $K_T$ tracks first-order autoregressive (AR(1)) partial-sum costs. For general matrix--matrix multiplication (GEMM), independently calibrated predictions differ from measurements by at most $3\%$ on the tested grid. A reduction tree extracted from an array library predicts the measured RMS scaling. However, stagnation and bias in positive low-precision sums limit the model's applicability.

\end{abstract}

{%
\renewcommand{\thefootnote}{\fnsymbol{footnote}}%
\footnotetext{\noindent\rule{0.8\columnwidth}{0.4pt}\\[2pt]This manuscript has been authored by UT-Battelle,
  LLC under Contract No.\ DE-AC05-00OR22725 with the
  U.S.\ Department of Energy. The publisher, by accepting the
  article for publication, acknowledges that the United States
  Government retains a non-exclusive, paid-up, irrevocable,
  world-wide license to publish or reproduce the published form
  of this manuscript, or allow others to do so, for United States
  Government purposes. The Department of Energy will provide
  public access to these results of federally sponsored research
  in accordance with the DOE Public Access Plan
  (\url{http://energy.gov/downloads/doe-public-access-plan}).}%
}

\section{Introduction}
\label{sec:intro}

Every parallel sum has a reduction tree: a binary tree whose leaves are the inputs and whose internal nodes represent the pairwise additions that combine them. Consider inputs $p_1,\ldots,p_k$ with exact sum $s=\sum_{i=1}^k p_i$, and let $\widehat s$ be the computed result in floating-point arithmetic with unit roundoff $u$. An implementation may choose a sequential, pairwise, blocked, or hybrid tree for locality, synchronization, or bandwidth. We call this tree structure the \emph{reduction geometry}. Although often chosen for performance, the geometry determines which intermediate values are rounded and how their errors accumulate. Two trees with the same leaves and operation count can therefore produce different errors. By contrast, the standard tree-independent forward-error bound ignores this geometry:
\[
|\widehat s-s|\leq \gamma_{k-1}\sum_{i=1}^k|p_i|,
\qquad \gamma_n:=\frac{n u}{1-n u},
\]
for $(k-1)u<1$. This bound provides a worst-case guarantee but no tree-specific mean-square estimate for random inputs. It therefore cannot rank a mixed library tree against a pairwise tree on random data; a second-moment model can. Our analysis keeps the geometry that the bound discards, as Figure~\ref{fig:tree-statistics} illustrates. Internal subtree sizes determine two scalar statistics, while shared ancestry determines a matrix-valued kernel.

\begin{figure*}[t]
\centering
\includegraphics[width=\textwidth]{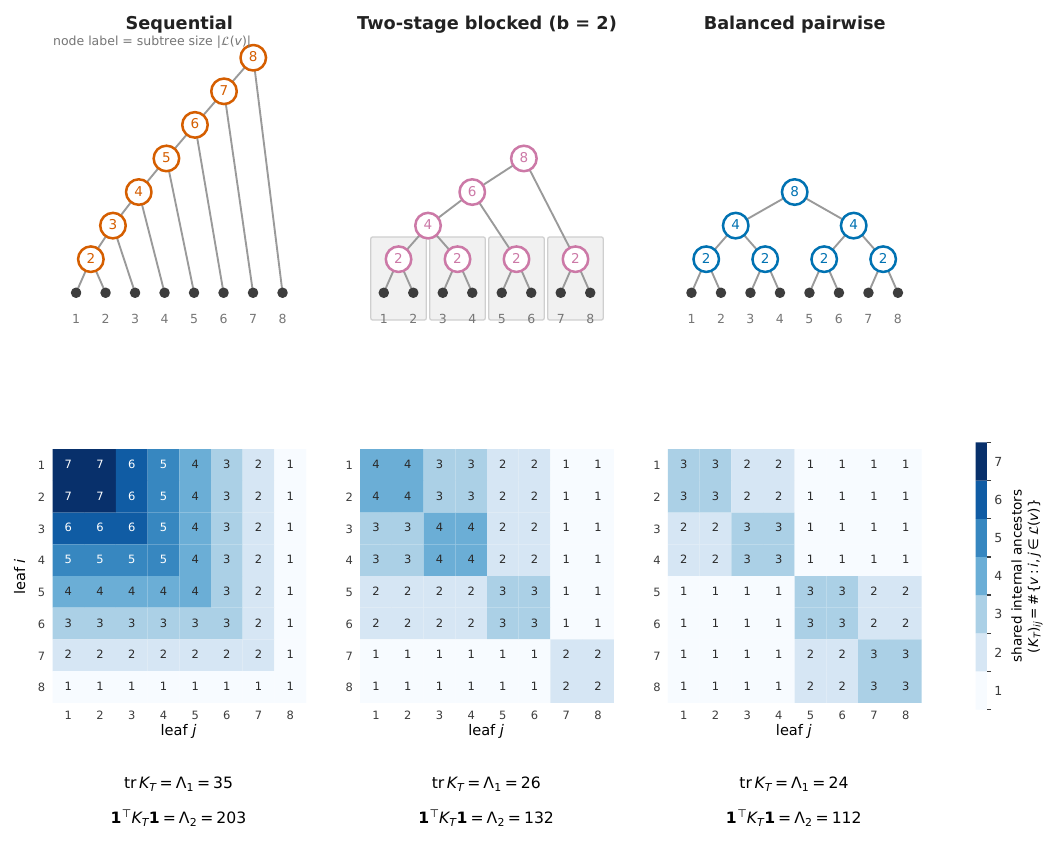}
\caption{Shared ancestry determines the second-moment cost for three reduction trees on the same $k=8$ inputs. Top: each internal-node label is its subtree size $|\mathcal L(v)|$; the blocked tree accumulates sequentially within blocks and across block totals. Bottom: $(K_T)_{ij}$ counts the internal ancestors shared by leaves $i$ and $j$. The displayed diagonal and all-entry sums give $\Lambda_1=\operatorname{tr}K_T$ and $\Lambda_2=\mathbf 1^\top K_T\mathbf 1$. Thus the sequential and blocked centered-RMS scales are $\sqrt{35/24}=1.21$ and $\sqrt{26/24}=1.04$ times the balanced-pairwise scale. Section~\ref{sec:ancestor-kernel} defines $K_T$, and Section~\ref{sec:iid-compression} derives the two statistics.}
\label{fig:tree-statistics}
\end{figure*}

On a computational tree, the sum of squared exact partial sums, rather than operation count alone, controls typical error growth \cite{HighamMary2019Probabilistic,HallmanIpsen2023PrecisionAware}. To make this quantity precise, write $p=(p_1,\ldots,p_k)^\top$. For each internal node $v$ of $T$, let $q_v$ be the exact sum of the leaves below $v$. The tree-dependent cost is $\sum_v q_v^2$, where the sum ranges over all internal nodes.

We represent this cost as a quadratic form using the \emph{common-ancestor kernel} $K_T$, the $k\times k$ matrix whose entry $(K_T)_{ij}$ counts the internal nodes that are ancestors of both leaves $i$ and $j$. Then
\[
\sum_vq_v^2=p^\top K_Tp.
\]
For random inputs with second-moment matrix $M:=\E_p[pp^\top]$, the expectation over $p$ yields the scalar contraction
\[
\E_p\!\left[\sum_vq_v^2\right]
=\E_p[p^\top K_Tp]
=\operatorname{tr}(K_TM)
=\langle K_T,M\rangle,
\]
where $\langle X,Y\rangle:=\operatorname{tr}(X^\top Y)$ is the Frobenius inner product. This contraction captures how the reduction geometry interacts with the input statistics. For i.i.d. inputs, it reduces to the two scalar statistics $\Lambda_1$ and $\Lambda_2$; the same construction extends to matrix multiplication.

To analyze the tree error $e_T:=\widehat s-s$, define $\operatorname{MSE}(e_T):=\E[e_T^2]$ and $\operatorname{RMS}(e_T):=\sqrt{\operatorname{MSE}(e_T)}$. We call inputs centered when their means are zero and noncentered otherwise. The analysis separates three layers of increasing specialization.

\emph{Layer 1 (exact recurrence).} Standard stochastic rounding randomly selects one of the adjacent floating-point values that bracket an operand, with probabilities chosen to preserve that operand in expectation. Because stochastic rounding is conditionally unbiased, its accumulated error satisfies an exact recurrence with operand-dependent local variances (Theorem~\ref{thm:main}).

\emph{Layer 2 (constant-$\nu$ model).} The exact recurrence depends on the pre-rounding value at every node. We replace each operand-dependent local variance by $\nu u^2x^2$, where $\nu$ is a dimensionless local second-moment coefficient. This substitution preserves the tree structure while reducing the variance model to one coefficient (Corollary~\ref{cor:constant-nu}).

\emph{Layer 3 (round-to-nearest calibration).} Under deterministic round-to-nearest, local errors do not generally have conditional mean zero. We therefore treat the same leading expression as an empirical MSE approximation over random inputs and calibrate $\nu$ from independent node-local errors. Section~\ref{sec:experiments} quantifies both its accuracy and its failure regimes. A fitted round-to-nearest value of $\nu$ serves as an empirical calibration parameter rather than an arithmetic constant.

\subsection{Prediction Workflow}
\label{sec:prediction-workflow}

For pure summation, the constant-$\nu$ model gives the following leading prediction on an extracted tree with i.i.d. summands of mean $\mu$ and variance $\tau^2$:
\begin{equation}
\label{eq:workflow}
\operatorname{RMS}(e_T)\approx
u\sqrt{\nu\{\tau^2\Lambda_1(T)+\mu^2\Lambda_2(T)\}}.
\end{equation}
Evaluating this prediction requires three ingredients: the tree topology, the input statistics, and an independent calibration. We organize them into the validation workflow in Table~\ref{tab:workflow}.

\begin{table}[t]
\centering
\caption{Prediction workflow for a fixed binary reduction tree.}
\label{tab:workflow}
\ra{1.15}
\begin{tabular}{@{}cl@{}}
\toprule
Step & Quantity \\
\midrule
1 & Extract $T$ and its arithmetic at every node. \\
2 & Accumulate $\Lambda_1,\Lambda_2$ in one tree traversal. \\
3 & Estimate $\mu,\tau^2$ (or $M$ for $\langle K_T,M\rangle$; \S\ref{sec:ancestor-kernel}). \\
4 & Calibrate $\nu$ independently for the arithmetic regime. \\
5 & Evaluate~\eqref{eq:workflow} and check model conditions. \\
\bottomrule
\end{tabular}
\end{table}

The same two-term decomposition identifies which tree statistic controls the prediction. Comparing its terms gives
\[
\begin{aligned}
\Lambda_1\text{ dominates if }&
\frac{|\mu|}{\tau}<\sqrt{\frac{\Lambda_1(T)}{\Lambda_2(T)}},\\
\Lambda_2\text{ dominates if }&
\frac{|\mu|}{\tau}>\sqrt{\frac{\Lambda_1(T)}{\Lambda_2(T)}}.
\end{aligned}
\]
For $k=4096$, the closed forms in Section~\ref{sec:iid-compression} give thresholds of approximately $0.019$ for a sequential tree and $0.038$ for a pairwise tree. Positive reductions such as softmax denominators, squared norms, and batch statistics can therefore enter the $\Lambda_2$-dominated regime even when their mean is modest relative to their standard deviation. The criterion supplies the tree-dependent rounding component of a precision policy that also accounts for arithmetic, conditioning, and performance. We develop this framework in stages: Section~\ref{sec:theorem} derives the exact recurrence and the common-ancestor kernel; Section~\ref{sec:exponents} optimizes the resulting tree statistics; Section~\ref{sec:gemm} extends the analysis to matrix multiplication; Section~\ref{sec:applications-context} maps implementations to trees; and Section~\ref{sec:experiments} validates the predictions experimentally.

Table~\ref{tab:notation} in \ref{app:notation} summarizes the core notation used throughout the scalar theory. We define symbols specific to hierarchy design, correlated inputs, GEMM, and calibration when they first appear.

\section{Conditional Second-Moment Model}
\label{sec:theorem}

Let $T$ be a full binary reduction tree with leaf values $p=(p_1,\ldots,p_k)^\top$. Write $\operatorname{int}(T)$ for its internal nodes and $h(T)$ for its height. For a node $v$, let $\mathcal L(v)$ be its leaf set, $q_v=\sum_{i\in\mathcal L(v)}p_i$ its exact partial sum, $\widehat q_v$ its computed value, and $e_v=\widehat q_v-q_v$ its accumulated error. Expectations over rounding variables for fixed $p$ are denoted by $\E_\xi[\cdot\mid p]$. Averaging over random inputs then uses $\E_p$. For deterministic round-to-nearest experiments, only $\E_p$ remains.

\subsection{An Exact Conditional Recurrence}

To derive the exact recurrence for conditionally unbiased rounding errors and control its cross terms, we condition on the information available just before node $v$ is rounded. We denote this information by the $\sigma$-algebra $\mathcal F_v^-$: it includes the fixed input $p$ and all rounding draws in the two child subtrees of $v$, but excludes the fresh draw at $v$ itself, hence the superscript ``$-$''. Conditioning on $\mathcal F_v^-$ fixes both the pre-rounding operand $x_v$ and the inherited error $x_v-q_v$. The cross term between that inherited error and the local error $\xi_v$ therefore vanishes in expectation. This conditioning underlies the exact recurrence; the remaining cross term between the two child errors vanishes by subtree independence.

\begin{definition}[Conditionally unbiased rounding errors]
\label{def:model}
Let $\operatorname{fl}_v$ denote the possibly random rounding map at node $v$. Let $\mathcal F_v^-$ be the $\sigma$-algebra generated by the fixed input $p$ and all rounding draws in the two child subtrees. With children $v_1,v_2$, the pre-rounding value
\[
x_v=\widehat q_{v_1}+\widehat q_{v_2}
=q_v+e_{v_1}+e_{v_2}
\]
is $\mathcal F_v^-$-measurable. The local error $\xi_v=\operatorname{fl}_v(x_v)-x_v$ satisfies
\[
\E_\xi[\xi_v\mid\mathcal F_v^-]=0,
\qquad
\E_\xi[\xi_v^2\mid\mathcal F_v^-]=\psi_v(x_v)\geq0.
\]
Thus $\psi_v(x)$ is the conditional local variance for operand $x$. Rounding draws in disjoint subtrees are independent conditional on $p$. Rounded product leaves, when present, obey analogous conditions and are independent of later accumulation draws.
\end{definition}

Keeping the operand- and node-dependent functions $\psi_v$ explicit preserves the exact stochastic-rounding recurrence. The constant-$\nu$ model below replaces them with homogeneous approximations.

\begin{theorem}[Exact conditional second-moment recurrence]
\label{thm:main}
Define $V_v=\E_\xi[e_v^2\mid p]$. Pure summation leaves have $V_i=0$; a rounded product leaf has $V_i=\E_\xi[(\xi_i^{(p)})^2\mid p_i]$. Every internal node satisfies
\begin{equation}
\label{eq:exact-recurrence}
V_v=V_{v_1}+V_{v_2}+\E_\xi[\psi_v(x_v)\mid p].
\end{equation}
\end{theorem}

\begin{proof}
Because $e_v=e_{v_1}+e_{v_2}+\xi_v$, conditioning first on $\mathcal F_v^-$ makes the cross term with $\xi_v$ vanish. The child errors have zero conditional means and are independent on disjoint subtrees, so $\E_\xi[e_{v_1}e_{v_2}\mid p]=0$. The remaining three second moments give~\eqref{eq:exact-recurrence}.
\end{proof}

Standard stochastic rounding supplies a concrete $\psi_v$. If $x\in[x_-,x_+]$ lies between adjacent floating-point numbers, $h=x_+-x_-$ and $\theta=(x-x_-)/h$, then
\begin{equation}
\label{eq:sr-local}
\E_\xi[\xi\mid x]=0,
\qquad
\psi(x)=\theta(1-\theta)h^2.
\end{equation}
Thus Theorem~\ref{thm:main} is exact for standard stochastic rounding, including the operand-dependent interpolation coordinate $\theta_v$.

\subsection{Constant-\texorpdfstring{$\nu$}{nu} Model}
\label{sec:constant-nu}

We now specialize the exact recurrence by replacing operand-dependent local variances with a homogeneous model. Pure sums round only additions, whereas inner products may also round products. Subscripts $p$ and $a$ distinguish these operations: $u_p,u_a$ are their unit roundoffs and $\nu_p,\nu_a$ are their nonnegative, dimensionless second-moment coefficients. Because relative error scales with the operand, the homogeneous model uses variance proportional to $u^2x^2$:
\[
\psi_v(x)=\nu_a u_a^2x^2,
\qquad
V_i=\nu_pu_p^2p_i^2,
\]
with separately calibrated product and accumulation coefficients.

For any node $w$, let $d(w)$ be its depth, measured as the number of edges from the root. Substituting the homogeneous model into~\eqref{eq:exact-recurrence} introduces the one-level propagation factor
\[
\omega=1+\nu_au_a^2.
\]
This factor measures how child-level variance is amplified as it passes through a rounded parent. Unrolling the recurrence from the root assigns each local term at $w$ the geometric weight $\omega^{d(w)}$. Under $h(T)u_a^2\ll1$, these weights differ from unity by only $O(h(T)u_a^2)$. For binary64, $\omega-1=\nu_au_a^2\approx1.23\times10^{-32}\nu_a$, so the weights are approximately one under the stated condition.

\begin{corollary}[Constant-$\nu$ model recurrence and solution]
\label{cor:constant-nu}
With the notation above,
\begin{align}
V_v&=\omega(V_{v_1}+V_{v_2})+\nu_au_a^2q_v^2,\label{eq:recurrence}\\
V_{\rm root}
&=\nu_pu_p^2\sum_{i=1}^k\omega^{d(i)}p_i^2
+\nu_au_a^2\sum_{v\in\operatorname{int}(T)}\omega^{d(v)}q_v^2.
\label{eq:weighted-solution}
\end{align}
If $h(T)u_a^2\to0$ and $\nu_a$ remains uniformly bounded, then
\begin{equation}
\label{eq:leading-order}
V_{\rm root}
=\left(
\nu_pu_p^2\sum_{i=1}^k p_i^2
+\nu_au_a^2\sum_{v\in\operatorname{int}(T)}q_v^2
\right)
\left[1+O\!\left(h(T)u_a^2\right)\right].
\end{equation}
\end{corollary}

\begin{proof}
Equation~\eqref{eq:recurrence} follows from
$\E_\xi[x_v^2\mid p]=q_v^2+V_{v_1}+V_{v_2}$. Expanding from the root assigns each local term the weight $\omega^{d(w)}$. Since $0\leq d(w)\leq h(T)$, this weight is $1+O(h(T)u_a^2)$ uniformly. Nonnegativity then gives~\eqref{eq:leading-order}.
\end{proof}

Theorem~\ref{thm:main} gives the exact stochastic-rounding recurrence. Corollary~\ref{cor:constant-nu} supplies the constant-$\nu$ model that Section~\ref{sec:experiments} evaluates as a round-to-nearest MSE approximation over random inputs.

\subsection{Common-Ancestor Kernel}
\label{sec:ancestor-kernel}

The quadratic cost $\sum_vq_v^2$ can be expressed as a quadratic form in the input vector. This representation defines the common-ancestor kernel. Specifically, the coefficient of $p_ip_j$ counts the internal nodes containing both leaves: their common internal ancestors. Figure~\ref{fig:tree-statistics} illustrates the resulting structure.

\begin{proposition}[Common-ancestor kernel]
\label{prop:ancestor-kernel}
Let $b_v\in\{0,1\}^k$ indicate the leaves below $v$, and define
\[
K_T=\sum_{v\in\operatorname{int}(T)}b_vb_v^\top,
\qquad
(K_T)_{ij}=\#\{v:i,j\in\mathcal L(v)\}.
\]
Then
\begin{equation}
\label{eq:ancestor-quadratic}
\sum_{v\in\operatorname{int}(T)}q_v^2=p^\top K_Tp.
\end{equation}
For the input second-moment matrix $M=\E_p[pp^\top]$,
\begin{equation}
\label{eq:ancestor-expect}
\E_p\!\left[\sum_vq_v^2\right]
=\operatorname{tr}(K_TM)=\langle K_T,M\rangle.
\end{equation}
\end{proposition}

\begin{proof}
Since $q_v=b_v^\top p$, summing $q_v^2=p^\top b_vb_v^\top p$ proves the first identity. Taking $\E_p$ and using the trace inner product proves the second.
\end{proof}

The kernel $K_T$ is a leaf-indexed specialization of graph-theoretic ancestral matrices \cite{AndriantianaEtAl2019Ancestral}: it counts internal common ancestors. Because a single traversal accumulates $\sum_vq_v^2$ and the scalar statistics below, $K_T$ need not be formed densely.

\subsection{I.i.d. Compression}
\label{sec:iid-compression}

For i.i.d. inputs, the kernel contraction simplifies to two scalar tree statistics because $M$ has only an identity and a rank-one component.

\begin{corollary}[I.i.d. tree statistics]
\label{cor:iid}
If $\E_p[p_i]=\mu$ and $\operatorname{Var}_p(p_i)=\tau^2$, define
\[
\Lambda_1(T)=\sum_{v\in\operatorname{int}(T)}|\mathcal L(v)|,
\qquad
\Lambda_2(T)=\sum_{v\in\operatorname{int}(T)}|\mathcal L(v)|^2.
\]
Then
\begin{equation}
\label{eq:iid-cost}
\langle K_T,M\rangle
=\tau^2\Lambda_1(T)+\mu^2\Lambda_2(T).
\end{equation}
Including rounded product leaves, the leading second moment under the constant-$\nu$ model is
\[
\nu_pu_p^2k(\tau^2+\mu^2)
+\nu_au_a^2\{\tau^2\Lambda_1+\mu^2\Lambda_2\}.
\]
\end{corollary}

Equivalently, $\Lambda_1=\operatorname{tr}K_T$ and $\Lambda_2=\mathbf 1^\top K_T\mathbf 1$ are the diagonal and all-entry sums displayed in Figure~\ref{fig:tree-statistics}. Figure~\ref{fig:two-invariants} shows where the variance and squared-mean terms exchange dominance for three tree families.

\begin{figure}[t]
\centering
\includegraphics[width=\columnwidth]{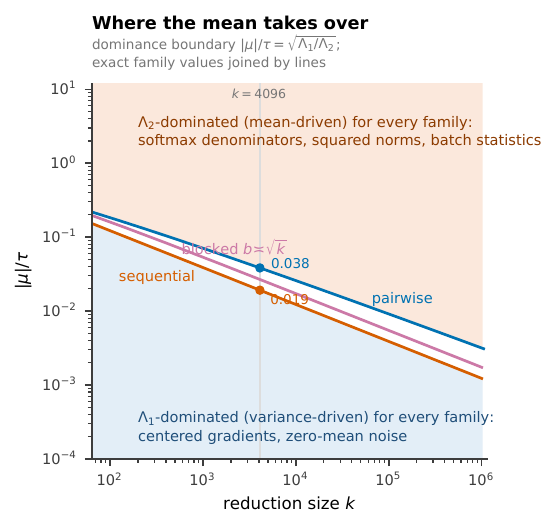}
\caption{Dominance boundary for i.i.d.\ pure summation. Each curve connects exact values of $|\mu|/\tau=\sqrt{\Lambda_1/\Lambda_2}$: integer sizes for sequential trees, powers of two for pairwise trees, and square sizes $k=b^2$ for two-stage blocked trees. Above a curve, the $\mu^2\Lambda_2$ term exceeds $\tau^2\Lambda_1$ for that family; below it, the reverse holds. The marked $k=4096$ thresholds are $0.019$ (sequential) and $0.038$ (pairwise). The workload labels illustrate typical regimes without imposing a distribution. Section~\ref{sec:iid-compression} derives the two terms, and Section~\ref{sec:prediction-workflow} states the criterion.}
\label{fig:two-invariants}
\end{figure}

For a two-stage blocked tree with sequential accumulation within blocks and across block totals, the formulas below apply to the divisible case $n_b=k/b$. For nondivisible cases, we evaluate $\Lambda_1$ and $\Lambda_2$ on the exact tree, including its shorter final block.

\begin{center}
\scriptsize
\ra{1.15}
\resizebox{\columnwidth}{!}{%
\begin{tabular}{@{}lll@{}}
\toprule
Tree & $\Lambda_1(T)$ & $\Lambda_2(T)$ \\
\midrule
Sequential & $k(k+1)/2-1$ & $k(k+1)(2k+1)/6-1$ \\
Pairwise, $k=2^j$ & $k\log_2k$ & $2k^2-2k$ \\
Two-stage blocked & $k^2/(2b)+kb/2+O(k)$
& $k^3/(3b)+kb^2/3+O(k^2+kb)$ \\
\bottomrule
\end{tabular}
}
\end{center}

The statistics have different meanings: $\Lambda_1$ is total leaf depth, while $\Lambda_2$ weights large subtrees quadratically. Their optimization therefore produces different schedules for centered and noncentered inputs.

\section{Tree Statistics and Optimization}
\label{sec:exponents}

We optimize $\Lambda_1$ and $\Lambda_2$ over several restricted families of reduction trees to identify structures that minimize the predicted RMS error in centered and noncentered regimes and connect these choices with classical blocking and scheduling results.

For centered i.i.d. inputs, the leading RMS scale is proportional to $\sqrt{\Lambda_1(T)}$. To formalize this scale asymptotically, let $\mathcal K\subseteq\mathbb N$ be an unbounded set of input sizes, and let $T_\bullet=(T_k)_{k\in\mathcal K}$ be a family with one $k$-leaf tree for each $k\in\mathcal K$. When the limit exists, we define the exponent $\alpha(T_\bullet)$ by
\begin{equation}
\label{eq:exponent-definition}
\alpha(T_\bullet)=\frac12\lim_{\substack{k\to\infty\\k\in\mathcal K}}
\frac{\log\Lambda_1(T_k)}{\log k}.
\end{equation}
The factor $1/2$ converts second-moment growth into an RMS exponent. This asymptotic exponent $\alpha$ differs from $\widehat\alpha$, a log--log slope fitted to a tree-statistic scale or measured RMS over a specified finite grid. Logarithmic and finite-size drift can separate them; pairwise RMS, for example, has exponent $1/2$ but includes a $\sqrt{\log k}$ factor. Interpreting the combinatorial prediction arithmetically requires $k\to\infty$, $u\to0$, and $h(T_k)u^2\to0$, with the local coefficient bounded above and below. In the subsections below, we derive exponents $1/2$, $2/3$, $3/4$, and $1$ for pairwise, three-stage, two-stage, and sequential families, respectively.

\subsection{Extremal Trees}
\label{sec:extremal-trees}

We first identify the trees that minimize and maximize $\Lambda_1$. A counting argument gives an alternative formula. Each internal node $v$ contributes its leaf count $|\mathcal L(v)|$ to $\Lambda_1$. Equivalently, each leaf $i$ is counted once for each of its $d_i$ internal ancestors. Summing over leaves gives
\begin{equation}
\label{eq:lambda1-depth}
\Lambda_1(T)=\sum_{i=1}^k d_i.
\end{equation}
The right-hand side is the tree's \emph{external path length}, the sum of all root-to-leaf path lengths. This formulation reveals the extremal structure: nearly balanced trees minimize the external path length, whereas the sequential tree maximizes it. Because the centered RMS scale is proportional to $\sqrt{\Lambda_1}$, the same ordering applies to the predicted error.

\begin{proposition}[$\Lambda_1$ extrema]
\label{prop:extremal}
Among full binary trees with $k$ leaves,
\begin{align*}
\Lambda_1^{\min}
&=k\lfloor\log_2k\rfloor
+2\{k-2^{\lfloor\log_2k\rfloor}\},\\
\Lambda_1^{\max}&=\frac{k(k+1)}2-1.
\end{align*}
Nearly complete balanced trees attain the minimum; the sequential tree attains the maximum. Consequently, every family with a well-defined exponent has $\alpha\in[1/2,1]$.
\end{proposition}

Every intermediate exponent is also attainable. To attain a target $\alpha\in(1/2,1]$, combine pairwise blocks of size $g=k^{2-2\alpha}$ sequentially. The outer stage contributes $\Theta(k^{2\alpha})$ to $\Lambda_1$, while the inner stages contribute $O(k\log g)$. The balanced family attains the lower extremum.

\subsection{Heterogeneous Variances}

When independent centered inputs have heterogeneous variances, the optimal tree minimizes a variance-weighted external path length. Specifically, Equation~\eqref{eq:ancestor-expect} becomes $\sum_i\tau_i^2d_i$. This objective is mathematically identical to the expected-codeword-length objective for a binary prefix code in which symbol $i$ has probability proportional to $\tau_i^2$ and codeword length $d_i$. The equivalence lets us import the classical Huffman optimality result directly.

\begin{proposition}[Huffman optimum and entropy bound]
\label{prop:huffman}
Assume independent centered inputs, an unconstrained full binary tree, and free leaf assignment. In the leading-order constant-$\nu$ regime, where $\omega^{d_i}=1+o(1)$ uniformly in $i$, Huffman coding with weights $\tau_i^2$ minimizes $\sum_i\tau_i^2d_i$ \cite{Huffman1952MinimumRedundancy,CoverThomas2006ElementsInformation}. If $W=\sum_i\tau_i^2$, $w_i=\tau_i^2/W$, and $H_2(w)$ is binary entropy, then
\[
WH_2(w)\leq\min_T\sum_i\tau_i^2d_i<W\{H_2(w)+1\}.
\]
\end{proposition}

Kao and Wang optimize a deterministic magnitude-weighted objective. For same-sign data, it reduces to weighted external path length with weights $|p_i|$; for mixed signs, it depends on signed partial sums \cite{KaoWang2000Summation}. Proposition~\ref{prop:huffman} instead optimizes the variance-weighted expected squared-partial-sum objective $\sum_i\tau_i^2d_i$ for independent centered inputs. With a maximum-depth constraint, the package-merge algorithm---a length-limited Huffman construction---gives the optimal tree \cite{LarmoreHirschberg1990LengthLimited}.

\subsection{Two-Stage Sequential Blocking}
\label{sec:two-stage-blocking}

In two-stage sequential blocking, each block and the block totals are summed sequentially. Here and below, $f(k)\asymp g(k)$ means that $c g(k)\leq f(k)\leq C g(k)$ for positive constants $c,C$ independent of $k$. With $b=k^\beta$,
\[
\Lambda_1\sim\frac{k^2}{2b}+\frac{kb}{2}
\asymp k^{2-\beta}+k^{1+\beta}.
\]
The first power decreases with $\beta$ and the second increases, so their maximum is V-shaped. Balancing the two terms yields the centered optimum.

\begin{proposition}[Centered blocked exponent]
\label{prop:blocked-exponent}
For $0\leq\beta\leq1$,
\[
\alpha_\tau(\beta)=\frac12\max(2-\beta,1+\beta).
\]
The two branches meet at $\beta=1/2$, where $b\asymp\sqrt{k}$ and $\alpha_\tau=3/4$.
\end{proposition}

Under the expected-RMS objective, Proposition~\ref{prop:blocked-exponent} recovers the balance $b\asymp\sqrt{k}$ found in classical worst-case analyses of blocked summation \cite{Higham1993AccuracySummation,FABsumBlanchardHighamMary2020} and in the two-stage member of the Superblock family \cite{CastaldoWhaleyChronopoulos2009Superblock}. This balance yields the exponent $\alpha_\tau=3/4$ for sequential accumulation within blocks and across their totals.

For noncentered inputs, the mean contribution instead uses
\[
\Lambda_2\sim\frac{k^3}{3b}+\frac{kb^2}{3},
\qquad
\alpha_\mu(\beta)=\frac12\max(3-\beta,1+2\beta).
\]
The exponent $\alpha_\mu(\beta)$ is minimized at $\beta=2/3$, where it equals $7/6$. The subscripts $\tau$ and $\mu$ denote the input-variance and input-mean contributions, respectively. Thus a single blocked construction requires different block growth rates in the $\Lambda_1$- and $\Lambda_2$-dominated regimes.

\subsection{Fixed-Stage Hierarchies}
\label{sec:fixed-stage-hierarchies}

We extend the optimization to fixed-stage hierarchies with $L$ sequential reduction levels and $L-1$ nested blocking levels. Let $r_\ell$ be the stage ratios, $\prod_{\ell=1}^Lr_\ell=k$, and let $s_\ell=\prod_{j\leq\ell}r_j$. Summing the stage contributions yields
\begin{equation}
\label{eq:multilevel-lambda}
\Lambda_1\sim\frac{k}{2}\sum_{\ell=1}^Lr_\ell,
\qquad
\Lambda_2\sim\frac{k}{3}\sum_{\ell=1}^Ls_\ell r_\ell.
\end{equation}

For centered inputs, the arithmetic--geometric mean inequality shows that $\Lambda_1$ is minimized when the stage ratios are equal, $r_\ell=k^{1/L}$. This choice yields
\begin{equation}
\label{eq:alpha-tau-L}
\Lambda_1^\star\sim\frac L2k^{1+1/L},
\qquad
\alpha_\tau(L)=\frac12\left(1+\frac1L\right).
\end{equation}
Optimizing $\Lambda_1$ independently recovers the centered geometric schedule selected by the Superblock optimizer \cite{CastaldoWhaleyChronopoulos2009Superblock}. Superblock controls worst-case exposure, whereas the schedule defined by~\eqref{eq:alpha-tau-L} minimizes the expected partial-sum-square statistic. The fixed-$L$ limit is distinct from the pairwise regime $L\asymp\log k$.

For noncentered inputs, however, the cumulative factors $s_\ell$ in~\eqref{eq:multilevel-lambda} make equal stage ratios suboptimal. We seek exponents $a_1,\ldots,a_L\geq0$ such that
\[
r_\ell=k^{a_\ell},
\qquad
\sum_{\ell=1}^L a_\ell=1.
\]
Set $A_0=0$ and $A_\ell=\sum_{j=1}^\ell a_j$. The $\ell$th term $k s_\ell r_\ell$ in~\eqref{eq:multilevel-lambda} scales as
\[
k^{E_\ell},
\qquad
E_\ell=1+A_{\ell-1}+2a_\ell.
\]
Thus the polynomial exponent of $\Lambda_2$ is $\max_\ell E_\ell$. We therefore choose the stage ratios to minimize $\max_\ell E_\ell$.

\begin{proposition}[Noncentered fixed-stage optimum]
\label{prop:noncentered-lstage}
The smallest polynomial exponent of $\Lambda_2$ over these feasible stage exponents is $1+(1-2^{-L})^{-1}$, attained by
\[
a_\ell=\frac{2^{-\ell}}{1-2^{-L}}.
\]
The corresponding RMS exponent is
\begin{equation}
\label{eq:alpha-mu-L}
\alpha_\mu(L)=\frac{2^{L+1}-1}{2^{L+1}-2}.
\end{equation}
\end{proposition}

\begin{proof}
Let $t$ bound every stage-dependent part, so that
$A_{\ell-1}+2a_\ell\leq t$. Since
$A_\ell=A_{\ell-1}+a_\ell$, this inequality is equivalent to
\[
A_\ell\leq\frac{A_{\ell-1}+t}{2}.
\]
Iterating from $A_0=0$ yields
$A_\ell\leq t(1-2^{-\ell})$. The feasibility condition $A_L=1$ therefore requires
$t\geq(1-2^{-L})^{-1}$. We attain this lower bound by requiring equality at every stage, which yields the stated $a_\ell$. The full $\Lambda_2$ exponent is $1+t$; halving it to pass from the second moment to RMS gives~\eqref{eq:alpha-mu-L}.
\end{proof}

Because later stages act on larger accumulated partial sums, the $\Lambda_2$ objective forces their ratios to shrink. At the optimum, the stage exponents satisfy $a_{\ell+1}=a_\ell/2$. Thus they halve from one level to the next, yielding
\[
r_\ell=k^{2^{L-\ell}/(2^L-1)}.
\]
Equivalently, the cumulative-exponent increments $A_\ell-A_{\ell-1}=a_\ell$ double when read from the last stage back toward the first. This factor-of-two pattern motivates the name \emph{doubling-gap schedule}. Within this fixed-stage family, the centered geometric schedule optimizes $\Lambda_1$, whereas the doubling-gap schedule optimizes the polynomial order of $\Lambda_2$. The sequential, optimally blocked two-stage, three-stage geometric, and pairwise families have centered exponents $1$, $3/4$, $2/3$, and $1/2$ (with a logarithmic factor), respectively.

Figure~\ref{fig:lambda1-tree-penalties} collects the exact centered growth curves and the corresponding centered-to-noncentered exponent changes.

\begin{figure*}[t]
\centering
\includegraphics[width=\textwidth]{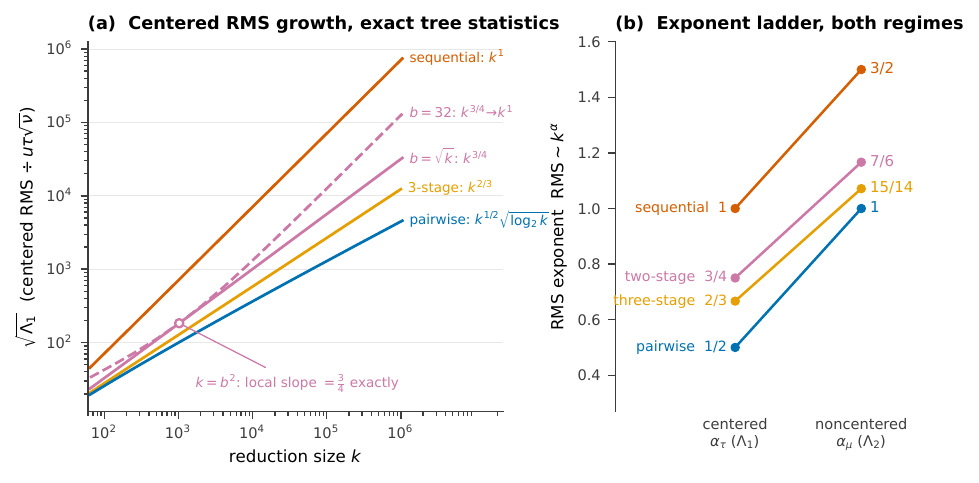}
\caption{Tree family and stage schedule jointly determine the RMS power law. Left: exact $\sqrt{\Lambda_1}$ for sequential, fixed-$b=32$ blocked, optimized two-stage blocked, three-stage geometric, and pairwise families. The fixed-$b$ curve has local slope $3/4$ at $k=b^2$ and approaches slope one, whereas $b\asymp\sqrt{k}$ retains exponent $3/4$. Right: centered exponents $\alpha_\tau$ from $\Lambda_1$ and noncentered exponents $\alpha_\mu$ from $\Lambda_2$; the two- and three-stage noncentered values use their doubling-gap optima. Equation~\eqref{eq:exponent-definition}, Proposition~\ref{prop:blocked-exponent}, and Proposition~\ref{prop:noncentered-lstage} yield the displayed values.}
\label{fig:lambda1-tree-penalties}
\end{figure*}

\section{General Matrix--Matrix Multiplication (GEMM)}
\label{sec:gemm}

We extend the scalar theory to general matrix--matrix multiplication (GEMM). Let $C=AB$ with $A\in\mathbb R^{m\times k}$ and $B\in\mathbb R^{k\times n}$. Each output entry is an inner product whose $k$ products are accumulated on the same tree $T$. The scalar recurrence applies entrywise, while a matrix formulation exposes how the operand data interact with the tree geometry. In particular, an algebraic identity expresses the exact-partial-matrix scale as a contraction of $K_T$ with a matrix $G$ formed from the operand Gram matrices. Combined with ensemble variance comparability, this Hadamard-product structure controls the expected squared Frobenius error.

\subsection{Exact Identity and Scaling Conditions}

By summing the scalar recurrence over output entries, we obtain an exact second-moment identity for GEMM without a scaling model.

\Needspace{15\baselineskip}
\begin{theorem}[Exact GEMM second moment]
\label{thm:gemm-exact}
Under Definition~\ref{def:model}, fix $A,B$. Let $\Sigma_p(A,B)$ sum the expected conditional variances of all rounded product leaves and let $\Sigma_a(A,B)$ sum those of all accumulation nodes. Then
\begin{equation}
\label{eq:gemm-exact-local}
\E_\xi\!\left[\|\widehat C-C\|_F^2\mid A,B\right]
=\Sigma_p(A,B)+\Sigma_a(A,B).
\end{equation}
\end{theorem}

\begin{proof}
Apply Theorem~\ref{thm:main} to every output entry and sum the conditional second moments.
\end{proof}

We need an ensemble-level condition to convert this identity into a two-sided $u^2x^2$ comparison. Such a comparison can fail when many nonzero operations are exact. For example, take $m=n=k$ and $A=B=I_k$ in binary arithmetic. Every product and partial sum is exactly representable, so all local variances vanish even though the squared product-leaf and exact-partial-matrix scales are positive. We exclude this aggregate degeneracy by requiring ensemble variance comparability.

\begin{assumption}[Ensemble variance comparability]
\label{ass:gemm-nondegenerate}
Let $(A,B)$ be drawn from an operand ensemble independently of the rounding draws. For $S_v\subseteq\{1,\ldots,k\}$, the set of inner indices below $v$, let $Q_v=A_{:,S_v}B_{S_v,:}$ be the exact partial matrix accumulated at node $v$. Define
\begin{align*}
\mathcal P(A,B)&=\sum_{\ell=1}^k
\|A_{:,\ell}\|_2^2\|B_{\ell,:}\|_2^2,\\
\mathcal A_T(A,B)&=\sum_{v\in\operatorname{int}(T)}\|Q_v\|_F^2.
\end{align*}
Here $\mathcal P$ is the total squared product-leaf scale and $\mathcal A_T$ the squared exact-partial-matrix scale. There exist positive comparison constants, independent of $m,n,k$, and $T$, such that
\begin{align*}
\E_{A,B}[\Sigma_p(A,B)]
&\asymp u_p^2\E_{A,B}[\mathcal P(A,B)],\\
\E_{A,B}[\Sigma_a(A,B)]
&\asymp u_a^2\E_{A,B}[\mathcal A_T(A,B)].
\end{align*}
When we do not separately round products, we set both $\Sigma_p$ and $u_p^2\E_{A,B}[\mathcal P(A,B)]$ to zero.
\end{assumption}

Here $X\asymp Y$ means $cY\leq X\leq CY$, with positive constants that may depend on the ensemble and arithmetic but not on $m,n,k,T$. The upper comparisons impose the usual $u^2x^2$ scale, while the lower comparisons exclude the aggregate degeneracy described above. Thus Assumption~\ref{ass:gemm-nondegenerate} makes the ensemble-and-arithmetic dependence explicit. In Section~\ref{sec:experiments}, we evaluate the resulting comparison on each stated operand ensemble and size grid using independent calibration.

\begin{corollary}[GEMM scaling form]
\label{cor:gemm-scaling}
Under Assumption~\ref{ass:gemm-nondegenerate},
\[
\begin{aligned}
\E_{A,B,\xi}\!\left[\|\widehat C-C\|_F^2\right]
&=\Theta\!\Bigl(\\[-0.4ex]
&\quad u_p^2\E_{A,B}[\mathcal P(A,B)]\\
&\quad+u_a^2\E_{A,B}[\mathcal A_T(A,B)]
\Bigr).
\end{aligned}
\]
\end{corollary}

\begin{proof}
Take $\E_{A,B}$ in Theorem~\ref{thm:gemm-exact} and apply Assumption~\ref{ass:gemm-nondegenerate}.
\end{proof}

\subsection{Kernel Contraction}

We now express the exact-partial-matrix scale as a contraction of $K_T$ with a quantity derived from two Gram matrices. The matrix $A^\top A$ is the column Gram matrix of $A$, and $BB^\top$ is the row Gram matrix of $B$. Their Hadamard product, denoted $\circ$, multiplies corresponding entries.

\begin{proposition}[GEMM common-ancestor contraction]
\label{prop:gemm-kernel}
Define
\[
G=(A^\top A)\circ(BB^\top),
\qquad
G_{\ell r}=\bigl((A_{:,\ell})^\top A_{:,r}\bigr)
\bigl(B_{\ell,:}(B_{r,:})^\top\bigr).
\]
Then
\begin{equation}
\label{eq:gemm-kernel}
\sum_{v\in\operatorname{int}(T)}\|Q_v\|_F^2
=\langle K_T,G\rangle,
\end{equation}
and the product-leaf scale is $\operatorname{tr}(G)$.
\end{proposition}

\begin{proof}
Expanding $\|Q_v\|_F^2$ yields $\sum_{\ell,r\in S_v}G_{\ell r}$. Summing over nodes counts each pair by its number of common ancestors, $(K_T)_{\ell r}$.
\end{proof}

Both Gram matrices are positive semidefinite. The Schur product theorem states that their Hadamard product is positive semidefinite, so $G\succeq0$. Because $K_T=\sum_vb_vb_v^\top$ is also positive semidefinite, $\langle K_T,G\rangle\geq0$. Forming $G$ densely requires $O(k^2(m+n))$ work and $O(k^2)$ storage, though structured representations can reduce this cost.

For fixed or correlated operands, evaluating the tree-dependent scale requires the full contraction $\langle K_T,G\rangle$ whenever the off-diagonal structure of $G$ is material. For independent centered entries with variance $\tau^2$, however, $\E_{A,B}[G]=mn\tau^4I_k$; the expected accumulation scale collapses to $mn\tau^4\Lambda_1(T)$ and can be evaluated without forming either dense matrix.

\subsection{Centered Random Matrices}

For centered random matrices with independent entries, the GEMM contraction simplifies and yields explicit RMS forward-error scales. Under Assumption~\ref{ass:gemm-nondegenerate}, entries with variance $\tau^2$ and finite fourth moment give
\begin{equation}
\label{eq:gemm-random}
\E_{A,B,\xi}\!\left[\|\widehat C-C\|_F^2\right]
=\Theta\!\left(mn\tau^4\{u_p^2k+u_a^2\Lambda_1(T)\}\right).
\end{equation}
For sequential fused multiply--add (FMA) accumulation, Equation~\eqref{eq:gemm-random} gives an RMS error relative to $\|C\|_F$ of order $u\sqrt{k}$.

Standard forward-error bounds use the cancellation-free scale $D=|A||B|$, with entrywise absolute values, so that
$D_{ij}=\sum_\ell|A_{i\ell}B_{\ell j}|$. This cancellation-free scale measures the magnitude exposed to summation even when terms in $C_{ij}$ cancel. An experiment produces a separate forward-error ratio on each trial, so we define
\[
\eta_F=\frac{\|\widehat C-C\|_F}{\|D\|_F}.
\]
The second-moment theory instead compares a ratio of RMS quantities,
\[
R_F:=\frac{
\bigl(\E_{A,B,\xi}\|\widehat C-C\|_F^2\bigr)^{1/2}}
{\bigl(\E_{A,B}\|D\|_F^2\bigr)^{1/2}}.
\]
Thus $R_F$ is a ratio of ensemble RMS values, with one population value for each $(m,n,k)$ configuration. By contrast, $\eta_F$ is a per-trial random variable that changes with each draw of $(A,B)$. The second-moment theory predicts $R_F$, not $\E[\eta_F]$, because the latter is the expectation of a ratio of correlated random variables. In the experiments, we report $R_F$ to validate the model and the distribution of $\eta_F$ to assess practical error.

\begin{proposition}[RMS forward-error scale]
\label{prop:gemm-rms-ratio}
Under Assumption~\ref{ass:gemm-nondegenerate}, let the entries of $A$ and $B$ be mutually independent copies of a centered random variable $X$ with variance $\tau^2>0$ and $\E|X|>0$. For sequential FMA accumulation with $u=u_a$,
\[
R_F=\Theta(u).
\]
\end{proposition}

\begin{proof}
For the matrix $D$ defined above, each entry satisfies
\begin{align*}
\E_{A,B}[D_{ij}^2]
&=k(\E|X|^2)^2+k(k-1)(\E|X|)^4\\
&=\Theta(k^2\tau^4).
\end{align*}
Hence $\E_{A,B}\|D\|_F^2=\Theta(mnk^2\tau^4)$, which has the same $k^2$ scale as the sequential accumulation term in~\eqref{eq:gemm-random} after multiplication by $u^2$.
\end{proof}

These normalizations yield the following forward-error scales for common inner-product accumulators.

\begin{center}
\footnotesize
\ra{1.15}
\begin{tabular}{@{}lll@{}}
\toprule
Tree & relative to $\|C\|_F$ & relative to $\||A||B|\|_F$ \\
\midrule
Sequential & $\Theta(u\sqrt{k})$ & $\Theta(u)$ \\
Blocked, $b\asymp\sqrt{k}$ & $\Theta(uk^{1/4})$ & $\Theta(uk^{-1/4})$ \\
Pairwise & $\Theta(u\sqrt{\log k})$ & $\Theta(u\sqrt{\log k/k})$ \\
\bottomrule
\end{tabular}
\end{center}

When off-diagonal entries contribute negligibly to $\langle K_T,G\rangle$, the tree-design objective reduces to the diagonal surrogate, and the Huffman construction from Section~\ref{sec:exponents} applies to $\operatorname{diag}(G)$. Otherwise, the objective remains the full contraction $\langle K_T,G\rangle$.

\section{Mapping Implementations to Trees}
\label{sec:applications-context}

To predict errors, we map an implementation to its arithmetic tree. Library kernels, vector lanes, accelerator blocks, and communication collectives may introduce different trees at different scales. FPRev, a tool that infers floating-point accumulation order from compiled binaries, can automate parts of this reconstruction \cite{XieEtAl2025FPRev}. In the NumPy study, we reconstruct the tree manually and validate it bitwise.

Once a tree is extracted, evaluating its statistics is inexpensive. Computing $\Lambda_1$ and $\Lambda_2$ requires one postorder traversal, while data-dependent analysis can accumulate $\sum_vq_v^2$ during the reduction. Dense representations of $K_T$, $M$, and $G$, however, cost $O(k^2)$, making them practical only for moderate-size or structured analyses.

\subsection{A Library Reduction}
\label{sec:numpy-tree}

To test the model on a production library, we follow the workflow in Table~\ref{tab:workflow}: we reconstruct the reduction tree used by NumPy's \texttt{np.sum} on a specific platform, compute its statistics, calibrate independently, and predict its error after bitwise validation.

We reconstructed NumPy~1.26.4 on an Apple M3 Max arm64 system running macOS~26.5.1 and Python~3.12.4. On this build, the inner kernel accumulated segments of fewer than eight terms sequentially and used eight unrolled accumulators within a 128-element block. For longer segments, recursive halving split the segment near its midpoint, aligned the split to eight elements, reduced both halves, and combined their totals. An outer sequence combined 8192-element working buffers from NumPy's universal-function (ufunc) machinery; \texttt{np.setbufsize} controls the default buffer size \cite{NumPyLoopsSource2024}.

The resulting four-scale hierarchy is
\begin{center}
\footnotesize
\begin{tabular}{@{}c@{\ $\longrightarrow$\ }c@{\ $\longrightarrow$\ }c@{\ $\longrightarrow$\ }c@{}}
\makecell{$<8$\\sequential} &
\makecell{$\leq128$\\eight lanes} &
\makecell{$\leq8192$\\recursive halving} &
\makecell{$>8192$\\buffer chain}
\end{tabular}
\end{center}

To validate this hierarchy, we compare reconstructed and observed bit patterns on fixed-seed vectors, vary the ufunc buffer size with \texttt{np.setbufsize}, and repeat versioned builds. Section~\ref{sec:experiments} reports the results of these validation steps and evaluates the error prediction on the extracted 1.26.4 tree. A version or platform change requires repeating both source reconstruction and bitwise validation.

\subsection{Communication and Accelerator Hierarchies}

For a scalar collective across $P$ ranks, a flat or ring-like arithmetic chain has sequential statistics, whereas a binomial or recursive-doubling tree has pairwise statistics. Message chunking and algorithm switching create nested trees; we extract their per-chunk arithmetic order to obtain the tree to which the formulas apply. We use the same extraction procedure for tensor-core blocks, single-instruction multiple-data (SIMD) lanes, GPU warp-shuffle instructions that exchange values among threads within a warp, and epilogue accumulation. The rounding law and internal accumulation order determine whether Definition~\ref{def:model} represents the resulting arithmetic.

\subsection{Workload Statistics}

Tree extraction and input modeling remain separate: after the tree is extracted, the input distribution determines the appropriate statistic. We use $\Lambda_1$ for approximately centered inner-product terms, both $\Lambda_1$ and $\Lambda_2$ for noncentered reductions, and $\langle K_T,M\rangle$ for correlated gradients or spatial data. Softmax denominators, sums of squares used in norms, and batch-normalization statistics are common positive reductions whose mean-to-standard-deviation ratio may exceed the threshold in Section~\ref{sec:intro}. Mixed-precision systems often retain selected reductions in binary32 \cite{MicikeviciusEtAl2018MixedPrecision}; the threshold supplies the tree-dependent rounding component of that precision-policy choice.

\section{Numerical Experiments}
\label{sec:experiments}

The exact recurrence~\eqref{eq:exact-recurrence} applies to conditionally unbiased stochastic rounding. The constant-$\nu$ model replaces operand-dependent local variances with constant coefficients and empirically approximates the MSE of deterministic round-to-nearest. Our experiments test structure and calibration separately: whether a tree statistic predicts how MSE changes, and whether an independently estimated coefficient predicts its level.

\subsection{Design and Reproducibility}
\label{sec:experiment-design}

Every reported summation error is the computed result minus its exact value. We use precision-specific strategies to evaluate these errors accurately. For binary64, Knuth's error-free \textsc{TwoSum} transformation records the exact residual at each node, and \texttt{math.fsum} combines those residuals at the roundoff-error scale. Python~3.12.4's standard-library \texttt{Decimal}, at 80-digit precision, independently checks selected cases.

For binary32, the exact sum of every node pair on the reported grids is representable in binary64. We therefore form its residual in binary64 before rounding to binary32. GEMM references also use 80-digit \texttt{Decimal}, and we check the system \texttt{libm} fused multiply--add (FMA) bitwise against one-round Decimal emulation. Bootstrap intervals resample trials independently at each size. The three-stage ($L=3$) ratios use a paired bootstrap because all schedules receive the same input. Table~\ref{tab:experiment-design} summarizes the grids.

We call an addition node \emph{stagnant} when its rounded output equals one operand while the other is nonzero, so that the addition does not change the retained operand.

\begin{table*}[t]
\centering
\caption{Experiment design. AR(1) denotes a first-order autoregressive input model. We fix trial counts in advance; all scripts, seeds, summaries, and compressed per-trial errors accompany the supplementary material.}
\label{tab:experiment-design}
\ra{1.12}
\begin{tabular}{@{}lllll@{}}
\toprule
Study & Arithmetic and inputs & Size grid & Trials & Primary comparison \\
\midrule
Centered scaling & binary64/binary32; uniform, normal & $64$--$8192$ & 800/size & $\widehat\alpha$ vs. $\sqrt{\Lambda_1}$ \\
Software low precision & binary16/bfloat16; centered, positive & $64$--$8192$ & 800/400 & MSE, bias, stagnation \\
Correlation & binary64 AR(1), $\rho\in\{-0.5,0,0.5,0.9\}$ & $k=64$ & 4000 & $K_T$ vs. i.i.d. model \\
Heterogeneous variance & binary64, two variance profiles & $k=256$ & 400 & Huffman/entropy bound \\
Noncentered schedules & binary64 normal$(1,1)$ & $k=2^{14}$ & 400 & paired L=3 ratios \\
Stochastic rounding & sums and $8\times k\times8$ GEMM & stated grids & 24--600 & exact recurrence \\
GEMM & binary64, $32\times k\times32$ & $32$--$256$ & 50 & independently calibrated MSE \\
Noncentered scaling & binary64; normal, uniform, chi-square & $64$--$2048$ & 300/size & finite-grid $\Lambda_1$--$\Lambda_2$ slopes \\
Coefficient calibration & binary64; centered uniform, normal & $\{64,256,1024\}$ & 600/size & $\nu_{\mathrm{eff}}$ \\
Local diagnostics & binary64; selected inputs and trees & selected $256$--$2048$ & 300--500 & cross terms; significands \\
Extracted NumPy tree & binary64; centered uniform & $2^7$--$2^{20}$ & \makecell[l]{1000 through $2^{15}$;\\300 thereafter} & extracted vs. pairwise tree \\
\bottomrule
\end{tabular}
\end{table*}

To isolate arithmetic effects without emulating a hardware kernel, we implement a software binary16/bfloat16 emulator with round-to-nearest-even and gradual underflow. We validate random binary16 values against NumPy and bfloat16 values against an independent bit-rounding implementation. Both emulator modes pass exact-value, midpoint-tie, subnormal, overflow, and Decimal residual-identity tests.

For pure-summation experiments, we define the effective coefficient
\[
\nu_{\mathrm{eff}}
:=\frac{\operatorname{MSE}(e_T)}
{u^2\,\E_p[\sum_{v\in\operatorname{int}(T)}q_v^2]}.
\]
For i.i.d. inputs, the factor multiplying $u^2$ in the denominator is $\tau^2\Lambda_1+\mu^2\Lambda_2$; for correlated centered inputs it is $\langle K_T,M\rangle$; and for heterogeneous independent centered inputs it is $\sum_i\tau_i^2d_i$. We define the squared-bias fraction as
\[
f_{\mathrm{bias}}:=
\frac{\{\E_p[e_T]\}^2}{\E_p[e_T^2]},
\]
and estimate it by dividing the squared sample mean by the sample mean square. Values near one indicate that systematic bias dominates MSE; values near zero indicate that trial-to-trial variation dominates.

\subsection{Centered Scaling Across Precisions}
\label{sec:exp-precision-scaling}

Figure~\ref{fig:precision-scaling} combines four precisions. For centered normal inputs, every precision preserves the ordering sequential $>$ blocked $>$ pairwise. The binary64 pairwise finite-grid theory slope is $0.579$ and the measured slopes are $0.581$ (normal) and $0.592$ (uniform); sequential measurements are $1.000$ and $0.980$. The blocked-$\sqrt{k}$ measurements lie near the predicted $0.750$. Binary32, software binary16, and software bfloat16 produce the same structural pattern. On this grid, the binary32 tests cover sampled uniform $[-1,1]$ and standard-normal operands.

\begin{figure*}[t]
\centering
\includegraphics[width=\textwidth]{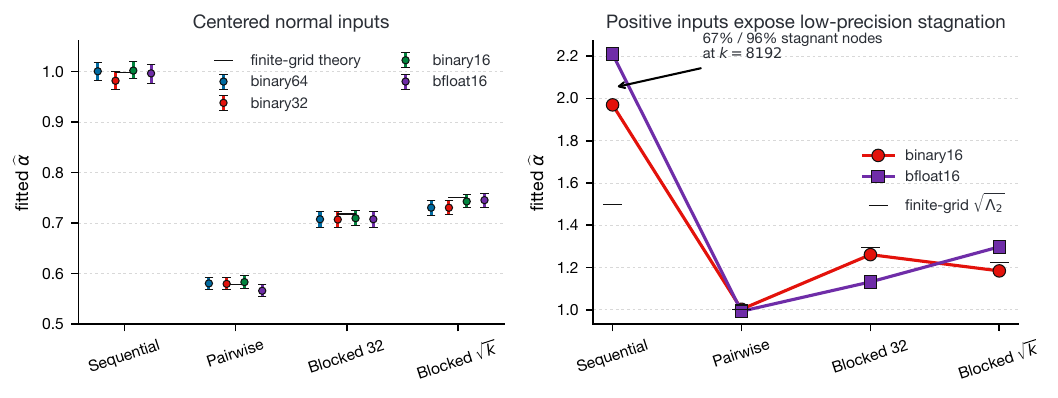}
\caption{Left: fitted centered-normal slopes with bootstrap 95\% intervals over $k=64,\ldots,8192$; black ticks are finite-grid fits of $\sqrt{\Lambda_1}$. Right: fitted slopes for uniform $[0,1]$ inputs in software binary16 and bfloat16; black ticks are the noncentered finite-grid slopes of $\sqrt{\Lambda_2}$. Sequential sums that stagnate at the largest size develop strong bias and depart from the constant-$\nu$ model. All blocked results use sequential accumulation within blocks and across block totals. Section~\ref{sec:exp-precision-scaling} discusses the topology ordering and low-precision stagnation boundary.}
\label{fig:precision-scaling}
\end{figure*}

The positive-input panel identifies the onset of low-precision stagnation. At $k=8192$, a sequential binary16 sum returns one operand unchanged at $66.6\%$ of internal nodes; bfloat16 reaches $95.8\%$. The corresponding squared-bias fractions exceed $0.9998$, and the fitted slopes rise to $1.97$ and $2.21$. By contrast, pairwise accumulation has negligible stagnation and remains near its noncentered exponent $1$. Thus, on this grid, low precision preserves the centered topology law. Monotone positive sums leave the leading-model regime when low precision causes stagnation and bias.

Across all centered configurations, the signed mean errors stored with the summaries contribute little to MSE. Large squared-bias fractions occur in the positive low-precision configurations.

\subsection{Blocked Families}
\label{sec:exp-blocked-families}

Fixed-$b$ and growing-$b$ predictions differ. For $b=32$, the local slope passes through $3/4$ near $k=b^2$ and moves toward one as larger sizes enter the fit window. When $b\asymp k^\beta$, the fitted points follow the V-shaped curve of Proposition~\ref{prop:blocked-exponent} (Figure~\ref{fig:blocked-scaling}). These results identify $3/4$ as the optimum within the two-stage sequential blocking family, not as a universal summation exponent.

\begin{figure*}[t]
\centering
\includegraphics[width=\textwidth]{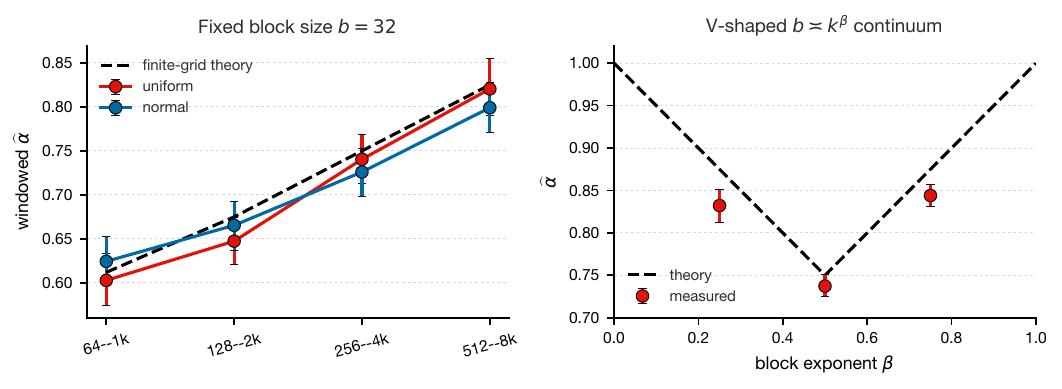}
\caption{Binary64 round-to-nearest, centered inputs. Left: four overlapping fit windows for fixed $b=32$. Right: block growth $b\asymp k^\beta$. Error bars are bootstrap 95\% intervals; dashed curves are finite-grid or asymptotic tree-statistic predictions. Section~\ref{sec:exp-blocked-families} discusses the fixed- and growing-block fits.}
\label{fig:blocked-scaling}
\end{figure*}

A separate four-stage experiment uses nine sizes $k=r^4$, $r=4,\ldots,12$, 1000 normal trials per size, and 2000 bootstrap replicates. Its fitted slope is $\widehat\alpha=0.602$ with interval $[0.591,0.614]$, compared with the exact finite-grid $\sqrt{\Lambda_1}$ slope $0.619$. \ref{app:l4} analyzes how $\nu_{\mathrm{eff}}$ varies.

\subsection{Correlated and Heterogeneous Inputs}
\label{sec:exp-correlated-heterogeneous}

Correlated inputs require the full kernel contraction, whereas heterogeneous independent variances lead to the Huffman objective.

For centered AR(1) inputs at $k=64$, the empirical partial-sum-square cost remains within $3.6\%$ of $\langle K_T,M\rangle$ for both sequential and pairwise trees (Figure~\ref{fig:ar1}). The i.i.d. statistic instead underpredicts the sequential cost by $14.2\times$ at $\rho=0.9$ and overpredicts it at $\rho=-0.5$. With the independent centered binary64 calibration, we predict the absolute RMS within $9.1\%$ across all eight configurations.

\begin{figure*}[t]
\centering
\includegraphics[width=\textwidth]{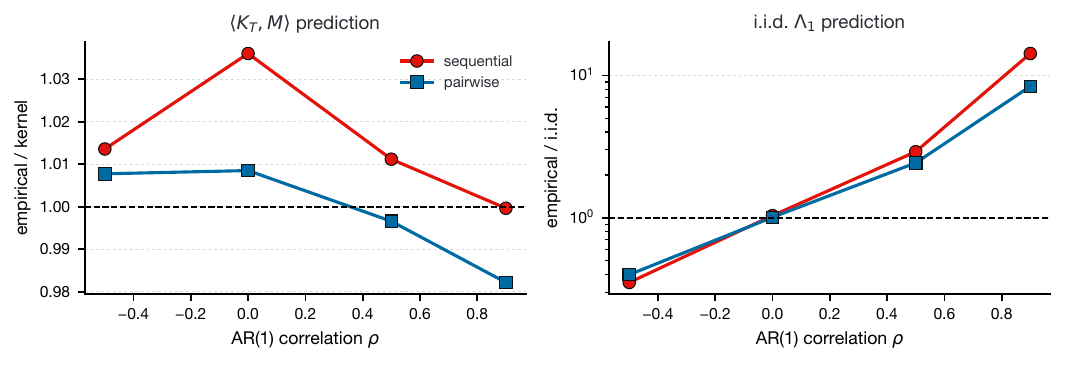}
\caption{First-order autoregressive (AR(1)) inputs. The full common-ancestor kernel tracks the empirical partial-sum cost (left), while the i.i.d. $\Lambda_1$ compression departs from the empirical cost as correlation increases (right, logarithmic scale). Section~\ref{sec:exp-correlated-heterogeneous} discusses the correlation test.}
\label{fig:ar1}
\end{figure*}

We test Proposition~\ref{prop:huffman} using heterogeneous independent variances at $k=256$. Figure~\ref{fig:huffman} normalizes every weighted leaf-depth cost by the entropy lower bound $WH_2(w)$. For $\tau_i^2\propto i^{-2}$, the Huffman, pairwise, and natural-order sequential ratios are $1.02$, $3.45$, and $109$; for $\tau_i^2\propto e^{-i/8}$, they are $1.01$, $1.80$, and $55.9$. Relative to Huffman, pairwise trees cost $1.79$--$3.37\times$ as much and sequential trees cost $55.5$--$106\times$ as much. Exact-reference round-to-nearest measurements give $\nu_{\mathrm{eff}}\in[0.160,0.232]$ across the six configurations, and the deterministic cost ordering persists across this coefficient range.

\begin{figure*}[t]
\centering
\includegraphics[width=\textwidth]{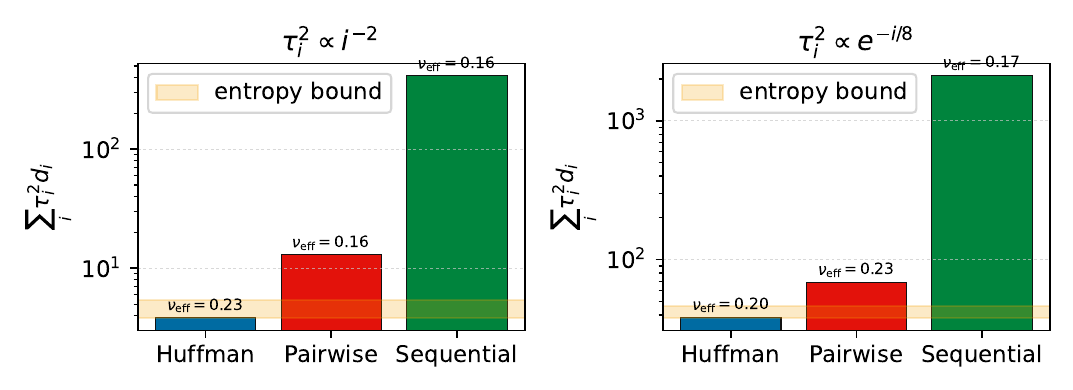}
\caption{Variance-weighted leaf-depth costs for two heterogeneous profiles ($k=256$). The shaded interval is the entropy bracket $[WH_2(w),W\{H_2(w)+1\}]$; labels above the bars give exact-reference round-to-nearest MSE coefficients. Proposition~\ref{prop:huffman} gives the bracket, and Section~\ref{sec:exp-correlated-heterogeneous} discusses the heterogeneous-variance comparison.}
\label{fig:huffman}
\end{figure*}

\subsection{Noncentered Inputs and Stage Schedules}
\label{sec:exp-noncentered-schedules}

Noncentered normal, uniform, and chi-square inputs shift the fitted slopes toward the $\Lambda_2$ predictions. Figure~\ref{fig:noncentered} compares measured and finite-grid slopes for normal$(1,1)$ and then tests the three-stage schedule optimizer. At $k=2^{14}$, the doubling-gap schedule $(256,16,4)$ has the smallest deterministic $\Lambda_2$. The geometric schedule has predicted RMS ratio $1.192$ and measured ratio $1.066$ with paired interval $[0.973,1.168]$. A nearby perturbation has predicted ratio $1.032$ and measured ratio $0.954$ with interval $[0.877,1.038]$. Both paired intervals contain unity, so the measurements support the predicted cost scale without statistically distinguishing the schedule ranking under round-to-nearest.

\begin{figure*}[t]
\centering
\includegraphics[width=\textwidth]{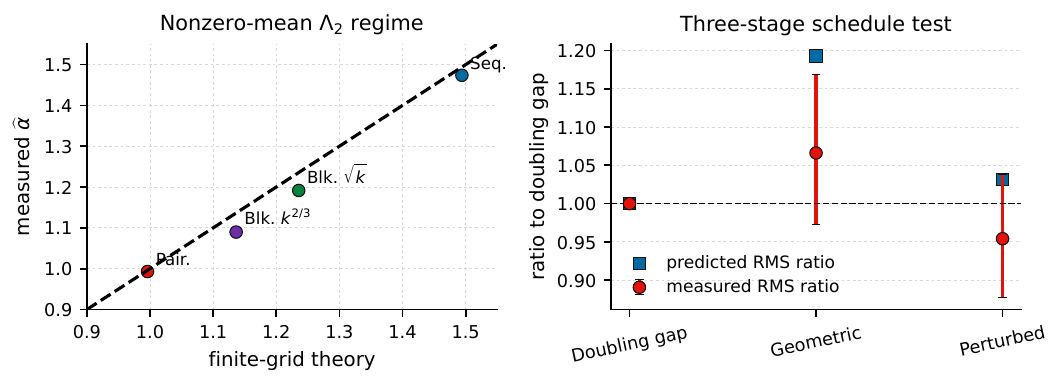}
\caption{Left: noncentered normal fitted slopes versus exact finite-grid tree-statistic slopes. Right: deterministic and measured three-stage schedule ratios. The predicted RMS ratio is $\sqrt{\Lambda_2/\Lambda_{2,\mathrm{dg}}}$; measured error bars are paired bootstrap 95\% intervals. Section~\ref{sec:exp-noncentered-schedules} discusses the slope and schedule comparisons.}
\label{fig:noncentered}
\end{figure*}

To determine whether deterministic correlations explain why the schedule rankings do not separate, we decompose the MSE as in \ref{app:auxiliary} on related diagnostic grids. Let $v<w$ enumerate each unordered node pair once. Across sequential, pairwise, and blocked-$\sqrt{k}$ trees at $k=256,1024,2048$, the cross terms $2\sum_{v<w}\E[\xi_v\xi_w]$ range from $-10.2\%$ to $+5.2\%$ of MSE. Their limited size does not by itself explain the unresolved ranking and instead leaves the configuration-dependent local coefficient as the larger observed source of variation. That appendix also reports the centered decomposition and supporting residual-identity tests.

\subsection{Exact Stochastic-Rounding Check}
\label{sec:exp-sr-check}

In Figure~\ref{fig:sr}, we compare measured second moments with the exact nodewise recurrence~\eqref{eq:exact-recurrence} using the local variance~\eqref{eq:sr-local}. Four of the six initial 95\% intervals contain unity. For the two initial misses, we conducted post hoc diagnostic runs at larger fixed sample sizes; these runs also yielded intervals containing unity. Because we selected the follow-ups after inspecting the initial results, we report the initial and diagnostic intervals separately; they do not form a combined confirmatory confidence statement.

\begin{figure}[t]
\centering
\includegraphics[width=\columnwidth]{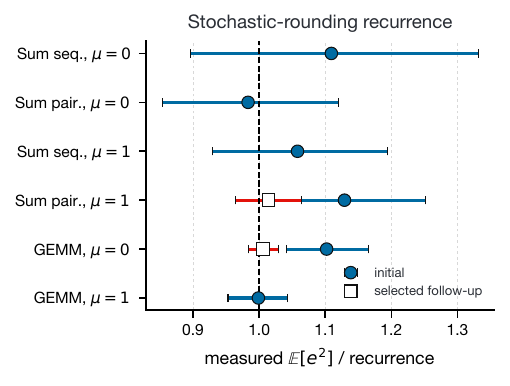}
\caption{Measured stochastic-rounding second moment divided by the exact nodewise recurrence. Circles show all six initial bootstrap 95\% intervals; open squares show post-hoc larger-sample diagnostics for the two initial misses. Section~\ref{sec:exp-sr-check} discusses the initial and diagnostic intervals.}
\label{fig:sr}
\end{figure}

\subsection{Calibrating \texorpdfstring{$\nu_{\mathrm{eff}}$}{nu-eff}}
\label{sec:nu}

For centered i.i.d. round-to-nearest summation, the definition in Section~\ref{sec:experiment-design} specializes to
\[
\nu_{\mathrm{eff}}=\frac{\operatorname{MSE}(e_T)}
{u^2\tau^2\Lambda_1(T)}.
\]
Across the binary64 calibration grid, the pooled mean is $0.1895$, while individual configurations span approximately $0.15$--$0.25$ (Figure~\ref{fig:nu}). We use $0.19$ as the calibration for centered binary64 configurations.
All coefficients use unit roundoff $u$; normalization by machine epsilon $\epsilon=2u$ would divide the reported coefficient by four.

To explain the spread in $\nu_{\mathrm{eff}}$, we examine representable spacing. Write $x=z2^e$ with significand $z\in[1,2)$; its \emph{binade} is the interval $[2^e,2^{e+1})$. Under the unit-roundoff convention used here, the spacing within this binade is $h=2u2^e$. Approximating roundoff as uniform on $[-h/2,h/2]$ gives variance $h^2/12$ and hence relative local MSE $u^2/(3z^2)$.

Averaging this $z^{-2}$ coefficient under uniform and logarithmic significand laws gives $1/6\approx0.167$ and $1/(8\ln2)\approx0.180$, respectively. Figure~\ref{fig:nu} (right) tests the cell-level prediction by normalizing out $z^{-2}$. The measured local second moments show the predicted dependence, with a mild downward trend across four bins. This mechanism explains the scale and configuration dependence, so the pooled coefficient remains configuration-dependent.

\begin{figure*}[t]
\centering
\includegraphics[width=\textwidth]{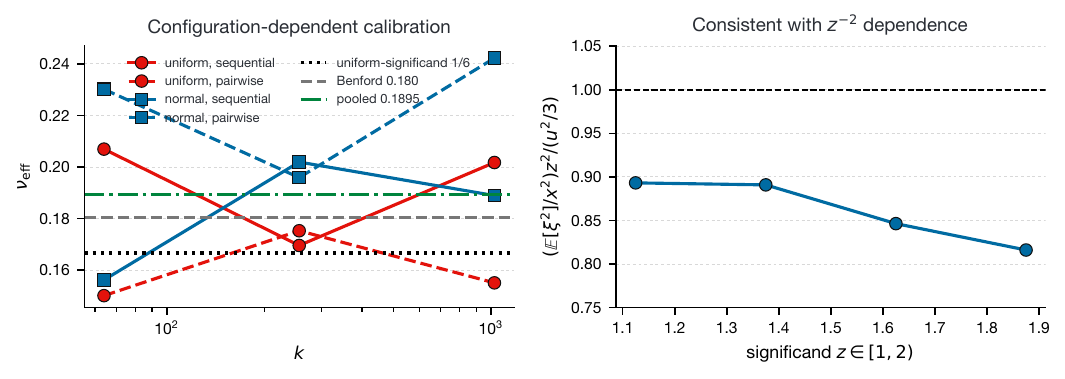}
\caption{Left: effective centered binary64 MSE coefficient across size, tree, and input distribution, with three reference levels. Right: normalized local second moment versus significand; a horizontal value of one is the uniform-in-cell prediction. Section~\ref{sec:nu} discusses the calibration spread and significand mechanism.}
\label{fig:nu}
\end{figure*}

\subsection{GEMM}
\label{sec:exp-gemm}

The FMA experiment uses $m=n=32$, matrices with independent standard-normal entries ($\tau=1$), and $k=32,64,128,256$. Independent local-error runs calibrate accumulation coefficients $\nu_a\in[0.176,0.184]$. A separate product-and-addition rounding comparison calibrates $\nu_p\in[0.176,0.184]$ and $\nu_a\in[0.180,0.186]$. All coefficients are calibrated from node-local errors independently of the root error used for validation.

For this $\tau=1$ ensemble, Figure~\ref{fig:gemm-validation} shows that
\[
D_{\rm model}=mn\,u^2\{\nu_p k+\nu_a\Lambda_1(T)\}
\]
predicts mean squared Frobenius error within $3\%$ for all eight configurations. The RMS ratio $R_F$ remains nearly constant in $k$, while the separate product rounding adds the predicted leaf contribution. Per-trial $\eta_F$ central 95\% ranges lie within $[4.70,5.75]\times10^{-17}$.

\begin{figure*}[t]
\centering
\includegraphics[width=\textwidth]{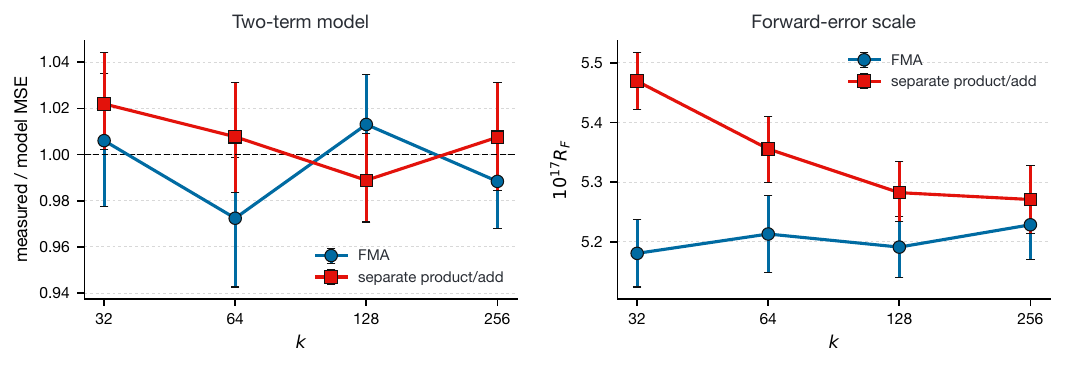}
\caption{Fused multiply--add (FMA) and separate product/addition rounding. Left: measured MSE divided by the independently calibrated two-term model. Right: $R_F$, the ratio of RMS quantities. Error bars are bootstrap 95\% intervals. Section~\ref{sec:exp-gemm} discusses the independent calibration and validation.}
\label{fig:gemm-validation}
\end{figure*}

For the corresponding non-FMA dot product, the product/accumulation crossover is the value of $k$ at which product rounding and accumulation contribute equally to the leading MSE. To calculate this crossover, we use independently calibrated constants rather than setting them equal. \ref{app:crossover} reports this auxiliary finite-size effect.

\subsection{Extracted NumPy Tree}
\label{sec:exp-numpy-tree}

Before testing MSE, we use intervention runs to verify the reconstructed tree. Changing \texttt{np.setbufsize} moves the outer boundary as predicted, and the NumPy~1.26.4 tree matches \texttt{np.sum} bit for bit on more than 1500 random vectors, including non-powers of two. Under NumPy~2.5.1, the observed sum matches the pure-pairwise reconstruction in all 20 trials at every tested size from $4096$ to $262144$, consistent with the contiguous-reduction fast path introduced during that development cycle \cite{Eendebak2026NumPyFastReduction}.

Figure~\ref{fig:npsum} evaluates Equation~\eqref{eq:workflow} on the reconstructed NumPy~1.26.4 node set. With $\nu=0.19$ fixed from the independent summation calibration, the model predicts the error within $7\%$ across $k=2^7,\ldots,2^{20}$. The extracted tree captures the bend caused by sequentially combined 8192-element buffers; a pure-pairwise model does not. The result applies to the specified NumPy~1.26.4 build and platform.

\begin{figure}[!b]
\centering
\includegraphics[width=\columnwidth]{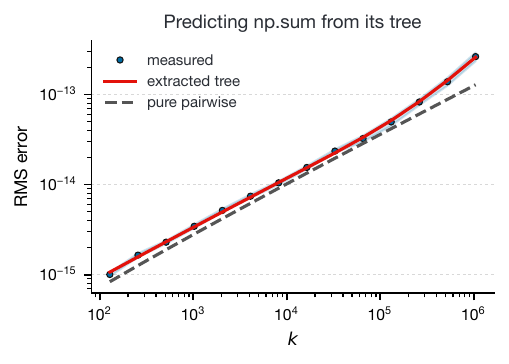}
\caption{Measured exact-reference RMS error of contiguous binary64 \texttt{np.sum}, the prediction from the extracted tree using the independently calibrated $\nu=0.19$, and the idealized pure-pairwise prediction. Section~\ref{sec:numpy-tree} sketches the extracted hierarchy. Shading is the bootstrap 95\% interval. Section~\ref{sec:exp-numpy-tree} discusses the validation and error prediction.}
\label{fig:npsum}
\end{figure}

\section{Related Work}
\label{sec:related}

\subsection{Stability and Accurate Summation}

Classical analyses distinguish recursive, pairwise, and compensated summation and provide worst-case forward or backward guarantees \cite{Higham1993AccuracySummation,Higham2002ASNA}. Compensated algorithms, expansions, and superaccumulators change the arithmetic to obtain higher accuracy or reproducibility \cite{Priest1991ArbitraryPrecision,DemmelHida2004FastAccurateSummation,OgitaRumpOishi2005AccurateSumDot,CollangeEtAl2015NumericalReproducibility}. FABsum combines blocked summation with an accurate final stage to obtain an operation-count-independent backward bound \cite{FABsumBlanchardHighamMary2020}. By contrast, we study the typical second moment of uncompensated accumulation on a fixed tree.

The Superblock family provides a structural parallel for hierarchical dot-product schedules, but under a worst-case exposure objective \cite{CastaldoWhaleyChronopoulos2009Superblock}. Its centered geometric optimizer coincides with our $\Lambda_1$-optimal fixed-stage hierarchy, whereas our second-moment objective additionally distinguishes the $\Lambda_2$-optimal doubling-gap schedule. Earlier ordering analyses for nonnegative inputs anticipate different quadratic and cubic mean-square contributions \cite{RobertazziSchwartz1988Ordering,Espelid1995Summation}.

\subsection{Probabilistic Roundoff}

Probabilistic models of roundoff date to von Neumann and Goldstine \cite{vonNeumannGoldstine1947Numerical}. Henrici, Hull and Swenson, and Barlow and Bareiss studied how roundoff errors propagate and how they are distributed \cite{Henrici1966Probabilistic,HullSwenson1966Tests,BarlowBareiss1985Distributions}. Chatelin and Brunet applied a probabilistic propagation model to the eigenvalue problem \cite{ChatelinBrunet1990Probabilistic}. Higham and Mary impose a centering assumption on roundoff errors, whereas we distinguish centered from noncentered input summands \cite{HighamMary2019Probabilistic}.

At the computational-tree level, Hallman and Ipsen derive the first-order functional $u\sqrt{\sum_vq_v^2}$ \cite{HallmanIpsen2023PrecisionAware}. We take its expectation under an input second-moment model to obtain $\langle K_T,M\rangle$. For i.i.d. inputs, this contraction reduces to $\Lambda_1$ and $\Lambda_2$. We then optimize those statistics and extend the contraction to GEMM. For stochastic rounding, local errors are conditionally unbiased by construction \cite{ConnollyHighamMary2021StochasticRounding,CrociEtAl2022StochasticRoundingSurvey}. El~Arar et al. give a general stochastic-rounding variance framework \cite{ElArarEtAl2023VarianceSR}. For reduction trees, we derive the common-ancestor kernel, its $\Lambda_1,\Lambda_2$ compression, and the resulting optimization and GEMM formulas.

\subsection{Ordering, Extraction, and Hardware}

Kao and Wang optimize deterministic magnitude-weighted objectives for same-sign summation and study the mixed-sign problem \cite{KaoWang2000Summation,KaoWang2001PrefixHuffman}. Those objectives differ from the variance-weighted depth and squared-partial-sum objectives here. ReproBLAS and ExBLAS target bitwise reproducibility or exact accumulation rather than a typical-error prediction \cite{DemmelNguyen2013FastReproducible,IakymchukEtAl2016ReproducibleMatmul}. Statistical MPI studies show that implementation topology changes observed errors \cite{PollardNorris2020MPIReduction}. FPRev demonstrates that the topology can be reconstructed \cite{XieEtAl2025FPRev}.

Tensor-core and block-FMA analyses model arithmetic inside particular accelerator units \cite{BlanchardEtAl2020TensorCores}. Mixed-precision training and stochastic-rounding studies motivate low-precision reductions \cite{MicikeviciusEtAl2018MixedPrecision,GuptaEtAl2015LimitedPrecision}. Recent work studies Horner evaluation and pairwise summation under limited-precision stochastic rounding \cite{ElArarEtAl2026LimitedPrecision}. We analyze conditionally unbiased stochastic rounding exactly and evaluate the constant-$\nu$ model under deterministic round-to-nearest-even in the binary16/bfloat16 experiments. These emulation experiments isolate precision effects; applying the framework to a deployed accelerator begins by extracting its arithmetic tree.

\section{Summary and Limitations}
\label{sec:conclusion}

The common-ancestor kernel combines a reduction tree with an input second-moment matrix to produce the expected partial-sum-square cost. For i.i.d. inputs, the resulting cost depends on the tree only through $\Lambda_1$ and $\Lambda_2$, which separate the centered and noncentered regimes. This representation yields extremal-tree, blocking, fixed-stage, and variance-weighted Huffman results and extends to GEMM through a Hadamard product of Gram matrices. Stochastic-rounding diagnostics test the exact recurrence. Under round-to-nearest, an independently calibrated constant-$\nu$ model predicts topology ordering and absolute MSE within the tested configurations. The low-precision and four-stage results identify regimes in which coefficient drift, bias, or stagnation produces departures from that approximation.

\subsection{Data and Code Availability}

An archived supplementary package contains the pinned environments, scripts, seeds, summary data, compressed per-trial errors, validation tests, figure-generation code, and reproduction instructions for all reported results.

\subsection{Limitations}

\begin{itemize}
\item The leading second-moment formulas characterize random inputs on fixed binary trees and complement worst-case methods for adversarial or data-dependent orderings.
\item The constant-$\nu$ model represents operand-dependent local variance through a calibrated coefficient. Under round-to-nearest, the experiments quantify its dependence on distribution, tree, correlation, and significand stream.
\item The implementation evidence covers CPU software arithmetic, a reconstructed NumPy kernel, and validated binary16/bfloat16 emulation with gradual underflow. Applying the framework to GPU, tensor-core, and MPI implementations begins with extracting and validating their arithmetic trees.
\item The leading expansion applies when $h(T)u^2\ll1$; the positive low-precision experiments identify stagnation and bias as the dominant corrections outside this regime.
\end{itemize}

The framework opens extensions to concentration bounds, compensated trees, combined $\Lambda_1$--$\Lambda_2$ optimization, structured GEMM contractions, and extracted multi-level hardware hierarchies.

\appendix
\section{Auxiliary Results}

\subsection{Product and Accumulation Crossover}
\label{app:crossover}

For a centered sequential non-FMA dot product, let $c_p$ and $c_a$ denote the independently calibrated product and accumulation coefficients with their common input-variance and unit-roundoff factors divided out. The leading MSE is proportional to
\[
c_p k+\frac{c_a}{2}k^2.
\]
Define the local RMS exponent as the instantaneous log--log slope
$\alpha_{\rm loc}(k):=\frac{\mathrm d\log\operatorname{RMS}(k)}{\mathrm d\log k}$. Taking the square root of the leading MSE and differentiating logarithmically, we obtain
\begin{equation}
\label{eq:general-crossover}
\alpha_{\rm loc}(k)=\frac12\frac{c_p+c_a k}{c_p+c_a k/2},
\qquad
k_{3/4}=\frac{2c_p}{c_a}.
\end{equation}
Dividing the pooled MSE coefficients by $u^2$ yields $c_p=0.1801$ and $c_a=0.1831$, giving $k_{3/4}=1.97$. Figure~\ref{fig:product-crossover} plots this curve alongside the three measured local-slope intervals. This finite-size crossover is distinct from the asymptotic $3/4$ optimum of two-stage sequential blocking.

\begin{figure}[t]
\centering
\includegraphics[width=\columnwidth]{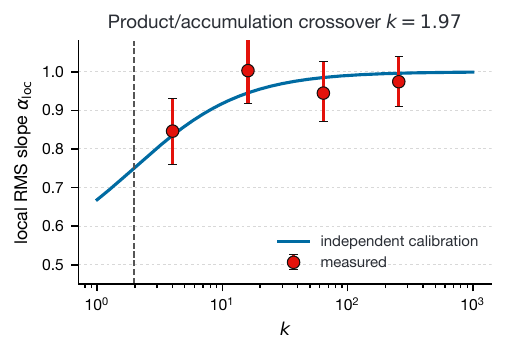}
\caption{Local non-FMA dot-product slopes and the curve from independently calibrated product and accumulation coefficients. The derivation appears in \ref{app:crossover}.}
\label{fig:product-crossover}
\end{figure}

\subsection{Four-Stage Scaling Experiment}
\label{app:l4}

The fixed-design four-stage experiment uses geometric trees with $k=r^4$ for $r=4,\ldots,12$, centered normal inputs, 1000 trials per size, 2000 bootstrap replicates, fixed seeds, and exact residual references. It records the exact $\Lambda_1$, the signed mean, the squared-bias fraction, $\nu_{\mathrm{eff}}$, and compressed per-trial errors. The fitted slope is $0.602$ with 95\% interval $[0.591,0.614]$, below the exact finite-grid $\sqrt{\Lambda_1}$ slope $0.619$ and the asymptotic value $0.625$ (Figure~\ref{fig:l4-followup}). Despite intermediate fluctuations, the effective coefficient $\nu_{\mathrm{eff}}$ decreases from $0.217$ at $k=256$ to $0.190$ at $k=20736$. This decrease shifts the finite-grid round-to-nearest slope. Because Equation~\eqref{eq:alpha-tau-L} depends only on the tree's combinatorics, the observed drift in $\nu_{\mathrm{eff}}$ on this grid explains the discrepancy.

\begin{figure}[t]
\centering
\includegraphics[width=\columnwidth]{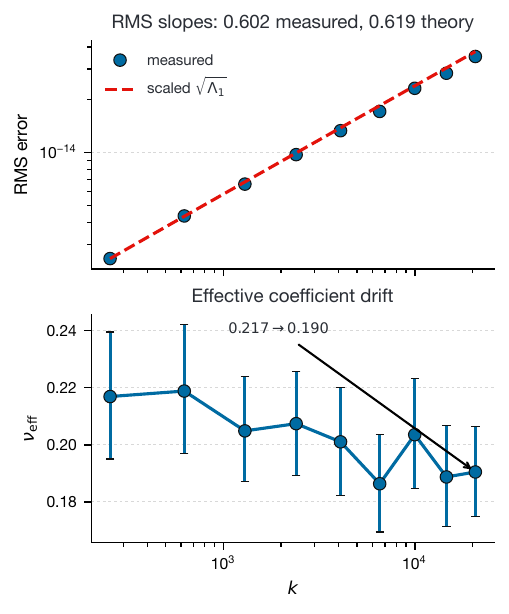}
\caption{Four-stage round-to-nearest scaling experiment. Left: measured RMS and the $\sqrt{\Lambda_1}$ curve scaled at the first point; their finite-grid slopes are $0.602$ and $0.619$. Right: $\nu_{\mathrm{eff}}=\operatorname{MSE}/(u^2\Lambda_1)$ with bootstrap 95\% intervals. Its decrease from $0.217$ to $0.190$ across the grid accounts for the observed slope difference. \ref{app:l4} describes the fixed-design experiment.}
\label{fig:l4-followup}
\end{figure}

\subsection{Local-Error Covariances and Unit Tests}
\label{app:auxiliary}

For deterministic round-to-nearest, the MSE expands to
\[
\operatorname{MSE}(e_T)=\sum_v\E_p[\xi_v^2]
+2\sum_{v<w}\E_p[\xi_v\xi_w],
\]
where $\sum_{v<w}$ ranges over each unordered pair of distinct internal nodes exactly once.
For centered $k=512$, the covariance contribution ranges from $-4\%$ to $+13\%$ across two distributions and two trees. For noncentered normal$(1,1)$ at $k=256,1024,2048$, it ranges from $-10.2\%$ to $+5.2\%$ across sequential, pairwise, and blocked-$\sqrt{k}$ trees. These bounded contributions support the local-variance approximation, while the remaining variation reflects the operand-dependent effective coefficient.

The unit tests cover both tree statistics and arithmetic. They reproduce every closed-form $\Lambda_1$ value exactly on the tested trees. The low-precision software tests compare binary16 with NumPy on 12,000 random values and bfloat16 with an independent bit-level reference on 10,000 binary32 values. Separate tests exercise midpoint, subnormal, and overflow boundaries and verify the complete reduction residual against Decimal arithmetic.

\section{Core Notation}
\label{app:notation}

Table~\ref{tab:notation} summarizes the core notation reused across the scalar theory, including the conditioning and asymptotic conventions. Symbols specific to hierarchy design, correlated inputs, significand calibration, communication costs, or GEMM are defined locally.

\begin{table}[!ht]
\centering
\caption{Core notation used across Sections~\ref{sec:theorem}--\ref{sec:exponents}.}
\label{tab:notation}
\ra{1.08}
\begin{tabular}{@{}p{0.32\columnwidth}p{0.60\columnwidth}@{}}
\toprule
Symbol & Meaning \\
\midrule
$T,\mathcal L(v),\operatorname{int}(T)$; $h(T),d(v)$ & tree, leaf set, internal nodes, height, and depth from the root \\
$p_i,q_v,\widehat q_v$ & exact input, exact partial sum, and computed partial sum \\
$\operatorname{fl}_v,\mathcal F_v^-,\psi_v$ & rounding map, pre-rounding information, and local variance \\
$e_v,\xi_v,V_v$ & accumulated error, local roundoff, and conditional second moment \\
$u_p,u_a$; $\nu_p,\nu_a$ & product and accumulation unit roundoffs and the corresponding local second-moment coefficients \\
$\omega$ & one-level variance propagation factor $1+\nu_au_a^2$ \\
$K_T,M$; $\langle\cdot,\cdot\rangle$ & common-ancestor kernel, input second moment, and trace inner product \\
$\Lambda_1,\Lambda_2$ & the sum of internal subtree sizes and the sum of their squares \\
$\mu,\tau^2$ & i.i.d. input mean and variance \\
$\alpha,\widehat\alpha,\alpha_{\rm loc}(k)$ & asymptotic exponent, finite-grid fit, and local log--log RMS slope \\
$f\asymp g$ & two-sided comparison up to positive size-independent constants \\
\bottomrule
\end{tabular}
\end{table}

\begin{acknowledgement}
This research was supported by the 2026 Laboratory Directed Research and
Development (LDRD) Program at Oak Ridge National Laboratory through the
initiative ``CCSD Core: Foundational Research for Smart Extreme-scale
Ecosystems.'' Oak Ridge National Laboratory is managed by UT-Battelle, LLC,
for the U.S.\ Department of Energy under Contract No.\ DE-AC05-00OR22725.
\end{acknowledgement}

\printbibliography

\end{document}